\documentclass[]{interact}

\usepackage{epstopdf}% To incorporate .eps illustrations using PDFLaTeX, etc.
\usepackage[caption=false]{subfig}% Support for small, `sub' figures and tables

\usepackage[numbers,sort&compress]{natbib}% Citation support using natbib.sty
\bibpunct[, ]{[}{]}{,}{n}{,}{,}% Citation support using natbib.sty
\theoremstyle{plain}% Theorem-like structures provided by amsthm.sty
\newtheorem{theorem}{Theorem}[section]
\newtheorem{lemma}[theorem]{Lemma}
\newtheorem{corollary}[theorem]{Corollary}

\newtheorem{claim}[theorem]{Claim} %doplnené

\theoremstyle{definition}

\newtheorem{example}[theorem]{Example}

\theoremstyle{remark}
\newtheorem{remark}{Remark}

\usepackage{amsmath}
\usepackage{kantlipsum}
\allowdisplaybreaks

\newcommand{\R}{\mathbb{R}}
\newcommand{\N}{\mathbb{N}}
\newcommand{\Q}{\mathbb{Q}}
\newcommand{\pp}{\mathbb{P}}
\newcommand{\ppp}{\pp^{-}}
\newcommand{\F}{\mathcal{F}}
\newcommand{\B}{\mathcal{B}\left(I\right)}
\newcommand{\nn}{\mathbb{N}_0}

\newcommand{\di}{\mathrm{diam}}

\allowdisplaybreaks

\newcommand{\dd}{\ \mathrm{d}}

\newcommand{\pd}{\left(p_n\right)_{n=0}^\infty}
\newcommand{\xx}{\{X_n\}_{n=0}^\infty}
\newcommand{\zz}{\{\hat{Z}_n\}_{n=1}^\infty}
\newcommand{\ff}{\{F_n\}_{n=0}^\infty}

\newcommand{\tn}{(T^n\hat{\mu})_j}

\begin{document}

%\articletype{ARTICLE TEMPLATE}% Specify the article type or omit as appropriate

\title{Discrete Dynamical Systems with Random Impulses}

\author{
\name{ J. Kováč\textsuperscript{a}\thanks{CONTACT Ján Veselý. Email: jan.vesely@fmph.uniba.sk}, J. Veselý\textsuperscript{a} and K. Janková\textsuperscript{a}}
\affil{\textsuperscript{a}Department of Applied Mathematics and Statistics, Faculty of Mathematics, Physics and Informatics, Comenius University in Bratislava, Mlynská dolina, Bratislava, Slovakia}
}

\maketitle

\begin{abstract}
We study the behaviour of discrete dynamical systems generated by a continuous map $f$ of a compact real interval into itself where at randomly chosen times a function different from $f$ - so called impulse function is applied. We show that both the  splittting property and the  average contraction  property guarantee the stability of the system. We give a number of examples where the verification of these properties is simple.
\end{abstract}


\section{Introduction}
It is well known that discrete dynamical systems generated by continuous self-mappings of a real compact interval $I$
\begin{align}
\label{DDS}
x_{n+1}=f(x_n),\text{ }x_0\in I
\end{align}
can exhibit very complex behavior that was described by various definitions of chaos. Recall Li and Yorke's chaos or distributional chaos \cite{SchweizerSmital} which allows one to compute a measure of chaos. A survey of definitions of chaos for continuous functions defined on a real interval can be found in \cite{ruette}. Discrete dynamical systems or parametric classes of these systems are often used in modelling in biology, physics or economy. Many models  naturally require to adopt a stochastic approach. As a consequence, a random dynamical system is obtained, which is easier to study if it satisfies the Markov property.
\par
It turns out that the chaos in system (\ref{DDS}) can be influenced by so-called impulses, where in certain times instead of the function $f$, which generates the system, the function $g$ (impulse) is applied. Such a procedure has been studied in \cite{DancaFeckan}, where the impulse times were chosen deterministically. This approach enables a so called control of chaos in dynamical systems which is interesting from the point of view of applications  and was discussed in \cite{GUANRONG}. Many papers in this field deal with impulsive differential equations, see e.g. \cite{YANG}, \cite{LI} or for random impulses see \cite{agarval}, \cite{sanz} and \cite{schmalfuss}. Motivated by this approach, we consider system (\ref{DDS})  where the impulse times are random. The distribution of times between impulses is an arbitrary discrete probability distribution. We will focus on the question of stochastic stability of the system. 
\par
The solution of this question is closely related to several results of papers, which have studied systems called iteration function systems with probabilities. The idea was that at any time either a function $f$ with probability $p$ or a function $g$ with probability $1-p$ was randomly applied, independently of the past. It is naturally possible to proceed to a random sample from a finite set of functions. The class of functions from which the sample is made is often the class of continuous monotone mappings of the interval into itself. In \cite{DubinsFreedman} a so-called splitting condition was given for such a class. The condition  guarantees the existence of a single limiting probability distribution. Another class of functions studied in this direction is the class of Lipchitz functions. Several contraction conditions have been studied for iterated function systems with probabilities generated by Lipschitz functions. A survey of these conditions can be found in \cite{stenflo2012}. Similar systems are studied as Markov chains in \cite{BM}, \cite{ohno}, \cite{BarnsleyDemko}, where their asymptotic behavior is considered, and from the point of view of the chaos also in \cite{KovacJankova}. 
\par
The abandonment of the condition of independence in the sample of functions has led to generalizations that assume that the process of sampling is controlled by a Markov chain. For example, in \cite{stenflo}, ergodic theorems are shown under the assumption that the functions satisfy the average contraction condition. In a recently published paper \cite{DiazMatias}, a finite set of functions that can be defined on a wide class of $m\text{-dimensional}$ metric spaces is considered. It assumes that a splitting condition generalizing the aforementioned condition from \cite{DubinsFreedman} is satisfied. The random dynamical systems generated under these assumptions  are studied via the so-called Markovian random products over the corresponding family of maps and the asymptotic stability of the corresponding Markov operator as well as the synchronization of the trajectories is shown. As we will show later, the results of \cite{DiazMatias} can be used to solve a problem with impulses, where the support of the time distribution between impulses is finite. In this paper,  however, we will assume that the distribution of times between impulses can be an arbitrary discrete probability distribution. Therefore we shall concentrate on the case where the support of this distribution is infinite countable. We formulate conditions which guarantee the stochastic stability of the corresponding random dynamical system. Some of the results from \cite{DiazMatias} can be used directly in the form from the paper, in the proof of other it was necessary to use a different approach because the support of the impulse times can take infinitely many values.

Let $f,g$ be continuous mappings from a real compact interval $I$ into itself. By an impulse at time $n$ in discrete dynamical system (\ref{DDS}) we mean that instead of the function $f$, the impulse function $g$ will be applied at time $n$. We will assume that the impulse times are given by independent and identically distributed discrete random variables $T_1, T_2, T_3,\ldots$, where $T_1$ represents the time (more precisely, the number of $f\text{'s}$) to the first impulse, and $T_n$ represents the time (number of $f\text{'s}$) between the $(n-1)\text{-th}$ and the $n\text{-th}$ impulse $g$ for $n\ge 2$. These random variables can take values from the set $\nn=\{0,1,2,\ldots\}$ and we denote the probability of taking the value $n\in\nn$ by $p_n$. The dynamics of the considered system is as follows - at the beginning we choose the point $x_0$ (either a point in $I$ or in general a random variable with respect to some measure $\mu_0$ on the measurable space $\left(I,\mathcal{B}(I)\right)$, where $\mathcal{B}(I)$ denotes the borel subsets of the interval $I$). Then we randomly generate the value of the random variable $T_1$ from the distribution $\pd$ - say $2$ with probability $p_2$ - that means, that $x_1=f(x_0)$, $x_2=f(x_1)$ and $x_3=g(x_2)$ (two $f\text{'s}$ and then impulse $g$). Subsequently, the value of random variable $T_2$ is generated (again from the distribution $\pd$, independently of $T_1$), say $0$, which means that $x_4=g(x_3)$ (zero $f\text{'s}$ between the first and the second impulse). We proceed in the same way for $T_3, T_4,\ldots$ to obtain the whole sequence~$\{x_n\}_{n=0}^\infty$. 

In general, we obtain a stochastic process $\{X_n\}_{n=0}^\infty$, and we will be interested in the conditions which guarantee some form of stability of this process. More precisely, if we denote by $\mu_n$ the distribution of the random variable $X_n$, $n\in\nn$,  we will study the conditions for the existence and uniqueness of a measure $\mu^\star$ on $(I,\mathcal{B}(I))$ such that $\mu_n\to\mu^\star$ weakly. Note  that the stochastic process $\xx$ need not be Markov - for example if $f(x)\equiv 1$ and $g(x)\equiv 0$ for every $x\in I$, then we get the following probabilities
%\begin{align*}
%Pr(X_4=0|X_3=1,X_2=1,X_1=0)&=Pr(T_1=0,T_2=2)=p_0p_2,\\
%Pr(X_4=0|X_3=1,X_2=0,X_1=1)&=Pr(T_1=1,T_2=1)=p_1^2. 
%\end{align*}
\small
\begin{align*}
Pr(X_3=0|X_2=1,X_1=0)&=\frac{p_1}{1-p_0},\\
Pr(X_3=0|X_2=1,X_1=1)&=\frac{p_2}{1-p_0-p_1}. 
\end{align*}
\normalsize
%\small
%\begin{align*}
%Pr(X_3=0|X_2=1,X_1=0)&=\frac{Pr(X_3=0,X_2=1,X_1=0)}{Pr(X_2=1,X_1=0)}=\frac{Pr(T_1=0,T_2=1)}{Pr(T_1=0,T_2\geq 1)}=\\
%&=\frac{p_0p_1}{\sum_{i=1}^{\infty}p_0p_i}=\frac{p_1}{1-p_0},\\
%Pr(X_3=0|X_2=1,X_1=1)&=\frac{Pr(X_3=0,X_2=1,X_1=1)}{Pr(X_2=1,X_1=1)}=\frac{Pr(T_1=2)}{Pr(T_1\geq 2)}=\\
%&=\frac{p_2}{\sum_{i=2}^{\infty}p_i}=\frac{p_2}{1-p_0-p_1}. 
%\end{align*}
%\normalsize
Thus, given $X_2$, the distribution of $X_3$ clearly depends also on the previous values. This complicates the study of the sequence of distributions $\{\mu_n\}_{n=0}^\infty$, since we cannot unambiguously determine the distribution $\mu_{n+1}$ from the distribution $\mu_n$. The problem can be solved using an approach of dynamical systems controlled by a Markov chain where a two-dimensional Markov process is defined for the study of the original one-dimensional system.

\subsection{Impulse system as a dynamical system controlled by a Markov chain}
Consider a Markov chain $\ff$ defined on the state space $\{0,1,2,\ldots\}$, where $F_n=0$ indicates that impulse function $g$ was applied at time $n$ and $F_n=k$ for $k\in\mathbb{N}$ signifies that the time to the next impulse is $k$. More specifically that the non-impulse function $f$ was applied at times $n, n+1,\ldots,n+k-1$ and the impulse function $g$ was applied at time $n+k$. For example, the event $\left\{T_1=2,T_2=0,T_3=3\right\}$ (i.e. the initial functions of the system are $f,f,g,g,f,f,f,g$) can be expressed as the event 
\begin{align*}
\{F_0=2, F_1=1, F_2=0, F_3=0, F_4=3, F_5=2, F_6=1, F_7=0\}. 
\end{align*}

The definition of $\left\{F_n\right\}_{n=0}^{\infty}$ implies that the chain moves from state $k\in\{1,2,\ldots\}$ to state $k-1$ with probability~$1$, i.e. we can write 
\begin{align*}
Pr(F_{n+1}=k-1|F_n=k)=1. 
\end{align*}
If the chain is in state $0$ (i.e. an impulse has just occurred), the time to the next impulse is chosen from the distribution $\pd$, hence
\begin{align*}
Pr(F_{n+1}=k|F_n=0)=p_k
\end{align*}
for $k\in\nn$.
The random variable $F_0$, as well as the random variable $T_1$, represents the time to the first impulse, therefore 
\begin{align*}
Pr(F_0=k)=p_k,
\end{align*}
for every $k\in\nn$.

Thus, the stochastic process $\ff$ is a Markov chain with the initial distribution~$\pd$ and the transition matrix
\begin{align}
\label{MPrechodu}
P=(p_{ij})_{i,j=0}^\infty=
\begin{pmatrix}
p_0 & p_1 & p_2 & \ldots \\
1 & 0 & 0 & \ldots \\
0 & 1 & 0 & \ldots \\
0 & 0 & 1 & \ldots \\
\vdots & \vdots & \vdots & \ddots
\end{pmatrix}.
\end{align}
Throughout the paper we will assume that 
\begin{itemize}
\item $p_0>0$ (to avoid the problems with periodicity),
\item the value $E\equiv E(T_1)=\sum\limits_{j=0}^\infty jp_j<\infty$ (to ensure that all states are positive recurrent),
\item the set $\{j\in \nn: p_j>0\}$ is not bounded - this is more of a technical point that will not significantly affect the claims in this paper - we add it because the case where it is bounded is addressed in \cite{DiazMatias}.
\end{itemize}
If we set $f_0\equiv g$ and $f_1=f_2=f_3=\ldots\equiv f$, the dynamical system with random impulses $\xx$ can be expressed as 
\begin{equation}
\label{dsimp}
X_{n+1}=f_{F_n}(X_n)
\end{equation}
for $n\in \nn$, where $X_0$ is a given starting point or a random variable. 

More formally, we will consider the set $\Sigma=\nn^{\nn}$ (an element $\omega\in\Sigma$ of this set is a sequence $\omega=~(\omega_0,\omega_1,\ldots)$ of elements of the set $\nn$) equipped with the cylindrical $\sigma\text{-algebra}$ $\mathcal{F}$, i.e. $\sigma\text{-algebra}$ generated by $n\text{-dimensional}$ cylinders 
\begin{align*}
\left[\xi_0,\ldots,\xi_{n-1}\right]=\left\{\omega\in \Sigma:\omega_0=\xi_0,\omega_1=\xi_1,\ldots, \omega_{n-1}=\xi_{n-1}\right\}.
\end{align*}
Further we define a probability measure $\pp^\star$ on $\mathcal{F}$ (based on the above-mentioned initial distribution $\pd$ and the transition matrix $P$) such that 
\begin{align*}
\pp^\star([\xi_0,\xi_1,\ldots,\xi_{n-1}])=p_{\xi_0}\cdot p_{\xi_0\xi_1}\cdot p_{\xi_1\xi_2}\cdot \ldots\cdot p_{\xi_{n-2}\xi_{n-1}}.
\end{align*}
%which can be extended to the whole $\sigma\text{-algebra}$ $\mathcal{F}$ using Kolmogorov extension theorem. 

Markov chain $\ff$ can be defined on the probability space $(\Sigma,\mathcal{F},\pp^\star)$ simply as $F_n(\omega)=\omega_n$ for any $n\in \nn$. Similarly, for a given $X_0=x_0$, the system with impulses $\xx$ can be viewed as a stochastic process on $(\Sigma,\mathcal{F},\pp^\star)$ defined as 
\begin{align*}
X_n(\omega)=f_\omega^n(x_0), 
\end{align*}
where $f_\omega^0(x_0)=x_0$ for every $\omega\in\Sigma$ and for $n\in\mathbb{N}$
\begin{align*}
f_\omega^n(x_0)=f_{\omega_{n-1}}\circ f_{\omega_{n-2}}\circ\ldots\circ f_{\omega_0}(x_0).
\end{align*}
Recursively, we can rewrite this formula as 
\begin{align}
\label{split}
X_{n+1}(\omega)=f_{\omega_n}(X_n(\omega))
\end{align}
for $n\in\nn$, which is in accordance with~(\ref{dsimp}).

As we have seen before, the process $\xx$ need not be Markov. However, a stochastic process $\{\hat{Z}_n\}_{n=1}^\infty$ defined as $\hat{Z}_n(\omega)=(\omega_{n-1},X_n(\omega))$ has the Markov property, because given $\hat{Z}_n$, the distribution of $\hat{Z}_{n+1}$ is independent of $\hat{Z}_{n-1},\hat{Z}_{n-2},\ldots,\hat{Z}_1$. This can be seen from the fact, that if $\hat{Z}_n=(k,x)$ for $k\in\mathbb{N}$ and $x\in I$, then $\hat{Z}_{n+1}=(k-1, f_{k-1}(x))$ with probability $1$ and if $\hat{Z}_n=(0,x)$, then $\hat{Z}_{n+1}=(j,f_j(x))$ with probability $p_j$ for every $j\in\nn$. 

Note that $\hat{Z}_n$ maps $\Sigma$ to the set $\hat{M}\equiv\nn\times I$, therefore the distribution of $\hat{Z}_n$ is a probability measure, say $\hat{\mu}^{(n)}$ defined on some appropriately chosen $\sigma\text{-algebra}$ $\hat{\mathcal{S}}$ over $\hat{M}$. Since $\zz$ is a Markov process, the distribution $\hat{\mu}^{(n+1)}$ of $\hat{Z}_{n+1}$ can be determined explicitly using the distribution $\hat{\mu}^{(n)}$ of $\hat{Z}_n$. Consider a set $\hat B\in \hat{\mathcal{S}}$ and its decomposition to 
\begin{align*}
\hat B = \bigcup_{j=0}^\infty \{j\}\times B_j, 
\end{align*}
where $B_j=\{x\in I: (j,x)\in \hat{B}\}$ for every $j\in\nn$. Then
\begin{align*}
\hat{\mu}^{(n+1)}(\hat{B})&=\pp^\star (\hat{Z}_{n+1}\in \hat{B})=\sum_{j=0}^\infty \pp^\star (\hat{Z}_{n+1}\in\{j\}\times B_j)=\\
&=\sum_{j=0}^\infty \left(\pp^\star(\hat{Z}_n\in\{j+1\}\times f_j^{-1}(B_j))+p_j\pp^\star(\hat{Z}_n\in \{0\}\times f_j^{-1}(B_j))\right)=\\
&= \sum_{j=0}^\infty \left(\hat{\mu}^{(n)}(\{j+1\}\times f_j^{-1}(B_j)) + p_j\hat{\mu}^{(n)}(\{0\}\times f_j^{-1}(B_j))\right).
\end{align*}
Moreover, for any measure $\hat{\mu}$ on $(\hat{M},\hat{\mathcal{S}})$ we can define for every $j\in \nn$ a measure $\mu_j$ on $(I,\mathcal{B}(I))$ as $\mu_j(C)\equiv \hat{\mu}(\{j\}\times C)$, $C\in \mathcal{B}(I)$. Clearly, 
\begin{align*}
\hat{\mu}(\hat{B})=\sum_{j=0}^\infty \mu_j(B_j). 
\end{align*}
Using this notation and the above relationship between $\hat{\mu}^{(n+1)}$ and $\hat{\mu}^{(n)}$, we get that the Markov operator~$T$ of the stochastic process $\zz$ is of the form
\begin{align}
\label{operator T}
T\hat{\mu}(\hat{B})=\sum_{j=0}^\infty\left(\mu_{j+1}(f_j^{-1}(B_j))+ p_j\mu_0(f_j^{-1}(B_j))\right). 
\end{align}

In this paper, we will show conditions for the weak convergence of the sequence of measures $\{T^n\hat{\mu}\}_{n=0}^\infty$ to a unique measure $\hat{\mu}^\star$, regardless of the initial measure $\hat{\mu}$. By projecting onto the second coordinate, we also obtain a weak convergence of distributions of the impulse process $\xx$. 

To formulate the main results, we first need to define two probability measures on the measurable space $(\Sigma, \mathcal{F})$ - a measure $\pp$ based on the stationary distribution of the Markov chain $\ff$ and a measure $\ppp$ based on the reversed Markov chain. 

\begin{remark}
Throughout the paper, we will often alternate between working with two-dimensional and one-dimensional objects. For easier differentiation, we will put a hat over the two-dimensional objects.
\end{remark}

\subsection{Invariant distribution and reversed chain}
It is easy to show that the Markov chain with the transition matrix $P$ has the invariant distribution $m\equiv (m_i)_{i=0}^\infty$, where 
\begin{align*}
m_i=\frac{\sum_{j=i}^\infty p_j}{1+E}
\end{align*}
for $i\in \nn$. Indeed, for any $j\in\nn$ we have
\begin{align*}
\sum_{i=0}^\infty m_i p_{ij} &= \sum_{i=0}^\infty \frac{\sum_{k=i}^\infty p_k}{1+E} p_{ij} \stackrel{(a)}{=}
\frac{\sum_{k=0}^\infty p_k}{1+E} p_{0j} + \frac{\sum_{k=j+1}^\infty p_k}{1+E} p_{(j+1)j} \stackrel{(b)}{=}\\
&\stackrel{(b)}{=} \frac{p_j}{1+E} +  \frac{\sum_{k=j+1}^\infty p_k}{1+E} =  \frac{\sum_{k=j}^\infty p_k}{1+E} = m_j,
\end{align*}
where in the equality $(a)$ we have used the fact that $p_{ij}\neq 0$ only for $i=0$ and $i=j+1$ and in the equality~$(b)$ we put $p_{0j}=p_j$ and $p_{(j+1)j}=1$. Note that under the assumption that the set $\{j\in \nn: p_j>0\}$ is not bounded, $m_i\neq 0$ for every $i\in \nn$. According to Theorems 8.1 and 8.2 in \cite{BM} on page 162 and 172, this distribution is unique and the distribution of the random variable $F_n(\omega)=\omega_n$ converges to $m$ in the total variation distance, i.e.
\begin{align*}
\lim\limits_{n\to\infty}\sum_{j=0}^\infty \left|Pr(F_n=j|F_0=i) - m_j\right|=0
\end{align*}
regardless of $i\in \nn$, hence the asymptotic behavior of the chain with the transition matrix $P$ does not depend on the initial distribution. 
Therefore, a sort of natural measure on $(\Sigma, \F)$ corresponding to the transition matrix $P$ is a probability
 measure $\pp$ whose finite-dimensional distributions have the form 
\begin{align*}
\pp\left([\xi_0,\ldots,\xi_{n-1}]\right)= m_{\xi_0}\cdot p_{\xi_0\xi_1}\cdot \ldots \cdot p_{\xi_{n-2}\xi_{n-1}}.
\end{align*}
In this case, the initial distribution is directly the invariant distribution, thus for the random variable $F_n(\omega)=\omega_n$ we have from the properties of the invariant distribution~that
\begin{align*}
\pp(F_n=i)=m_i
\end{align*}
for every $i,n\in\nn$.

Finally, in this paper we will work with the probability measure $\ppp$, also defined on $(\Sigma, \F)$ given by finite-dimensional distributions 
\begin{align*}
\ppp\left([\xi_0,\ldots,\xi_{n-1}]\right)= m_{\xi_0}\cdot q_{\xi_0\xi_1}\cdot \ldots \cdot q_{\xi_{n-2}\xi_{n-1}}, 
\end{align*}
where $q_{ij}\equiv \frac{m_j}{m_i}p_{ji}$ for $i,j\in\nn$. Such a measure corresponds to the so-called reversed Markov chain with the initial distribution $(m_i)_{i=0}^\infty$ and the transition matrix $Q=~(q_{ij})_{i,j=0}^\infty$. Note, that 
\begin{align*}
\ppp([\xi_0,\ldots,\xi_{n-1}])&=m_{\xi_0}\cdot q_{\xi_0\xi_1}\cdot \ldots \cdot q_{\xi_{n-2}\xi_{n-1}} 
=\\
&=m_{\xi_0}\cdot \frac{m_{\xi_1}}{m_{\xi_0}}\cdot p_{\xi_1\xi_0}\cdot \frac{m_{\xi_2}}{m_{\xi_1}}\cdot p_{\xi_2\xi_1}\cdot\ldots\cdot \frac{m_{\xi_{n-1}}}{m_{\xi_{n-2}}}\cdot p_{\xi_{n-1}\xi_{n-2}}=\\
&=m_{\xi_{n-1}}\cdot p_{\xi_{n-1}\xi_{n-2}}\cdot\ldots\cdot p_{\xi_2\xi_1}\cdot p_{\xi_1\xi_0}= \pp([\xi_{n-1},\ldots,\xi_0])
\end{align*}
(therefore ,,reversed chain'') and for any $j\in\nn$
\begin{align*}
\sum_{i=0}^\infty m_i q_{ij}&= \sum_{i=0}^\infty m_i\cdot\frac{m_j}{m_i}\cdot p_{ji}=m_j\sum_{i=0}^\infty p_{ji}=m_j, 
\end{align*}
thus $m$ is also  invariant distribution of the reversed chain. 

\subsection{Main results}
The properties of continuous functions on a compact interval imply that for any $\xi\in \Sigma$ the set 
\begin{align*}
I_\xi^n \equiv f_{\xi_0}\circ f_{\xi_1}\circ\ldots\circ f_{\xi_{n-1}}(I)
\end{align*}
is a closed interval or a singleton and $I_\xi^{n+1}\subseteq I_{\xi}^{n}$. Consequently, the set 
\begin{align*}
I_\xi\equiv \bigcap_{n=1}^\infty I_\xi^n
\end{align*}
is also a closed interval or a singleton. Denote by $S$ the set of $\xi\in\Sigma$ such that $I_\xi$ is a singleton, i.e.
\begin{align*}
S \equiv\{\xi\in\Sigma: \di\left(I_\xi\right)=0\}, 
\end{align*}
where $\di\left(I_\xi\right)$ denotes the diameter of the set $I_\xi$. Finally, consider a function $\pi:~S\to~I$ given by the formula 
\begin{align*}
\pi(\xi)=\lim_{n\to\infty} f_{\xi_0}\circ\ldots\circ f_{\xi_n}(x), 
\end{align*}
where $x\in I$ is arbitrary (the limit is the same for every $x\in I$ by the definition of $S$). 

\begin{theorem}
\label{thm1}
If $\ppp(S)=1$, then for any probability measure $\hat{\mu}$ defined on $(\hat{M},\hat{\mathcal{S}})$ and for any continuous and bounded function $\hat{h}:\hat{M}\to \R$
\begin{equation}\label{mainEq}
\lim_{n\to\infty} \int\limits_{\hat{M}} \hat{h}\dd T^n\hat{\mu}=\int\limits_{\hat{M}} \hat{h}\dd\hat{\pi}\odot\ppp, 
\end{equation}
where $\hat\pi:S\to \hat{M}$ is a function given by $\hat{\pi}(\xi)=(\xi_0,\pi(\xi))$ and $\hat{\pi}\odot \ppp$ denotes a pushforward measure on $(\hat{M},\hat{\mathcal{S}})$ defined~as
\begin{align*}
\hat{\pi}\odot \ppp(\hat{B})=\ppp(\hat{\pi}^{-1}(\hat{B}))
\end{align*}
for $\hat{B}\in\hat{\mathcal{S}}$.
\end{theorem}

Let's mention that in some cases we will integrate the functions $\pi$ or $\hat{\pi}$ over the whole $\Sigma$, but the measure of the set $\Sigma\setminus S$ will be zero in these cases, and thus it is possible to formally define these functions $\pi$ and $\hat{\pi}$ on this set arbitrarily.

%\begin{poz}
%\label{poz1}
%For any probability measure $\hat{\mu}$ on $(\hat{M},\hat{\mathcal{S}})$ we have $T^n\hat{\mu}\to \hat{\pi}\odot\ppp$ weakly. Moreover, the distribution $\hat{\pi}\odot\ppp$ is the unique stationary distribution of the stochastic process $\left\{Z_n\right\}_{n=1}^{\infty}$ . 
%\end{poz}

Since the condition $\ppp(S)=1$ is not easy to verify, we give two sufficient conditions for its validity.

\begin{theorem}
\label{thm2}
Let $f_0,f_1,\ldots$ be Lipschitz continuous with the Lipschitz constants $L_0, L_1,\ldots$. If the functions are contractive in average with respect to $\ppp$ in the sense that 
\begin{align}
\label{aveCont}
E_{\ppp} (\log L_{F_0})<0, 
\end{align}
then $\ppp(S)=1$.
\end{theorem}
Note, that in the paper \cite{stenflo} it is directly proved that under the assumption of average contraction~(\ref{aveCont}), the distribution of the dynamical system controlled by Markov chain converges weakly to a unique distribution. However, for the sake of completeness, we will provide a proof using an intermediate step with the set $S$. Let us also notice that in our case, where $f_1=f_2=\ldots$ (therefore $L_1=L_2=\ldots$), the expectation on the left side of the inequality~(\ref{aveCont}) can be simplified to
\begin{align*}
E_{\ppp}(\log L_{F_0})&=\sum_{i=0}^\infty (\log L_i)\ppp(F_0=i)=(\log L_0)m_0 + (\log L_1)\sum_{i=1}^\infty m_i=\\&=\frac{1}{1+E}\log L_0+\frac{E}{1+E}\log L_1.
\end{align*}
The obtained expression is less than zero if 
\begin{align}
\label{claim1}
%\frac{1}{1+E}\log L_0+\frac{E}{1+E} \log L_1<0,\\
%E\log L_1<-\log L_0,\\
L_1<L_0^{-E}.
\end{align}

This gives a sufficient condition for $\ppp(S)=1$ which can be simply verified. The second sufficient condition for $\ppp(S)=1$ assumes that the functions generating the dynamical system controlled by Markov chain are monotone injections that satisfy so called splitting condition as formulated in \cite{DiazMatias}. The splitting condition is satisfied if there exist $\pp$- admissible sequences $a_1,\ldots, a_l\in\nn$ and $b_1,\ldots, b_r\in\nn$ such that 
\begin{align}
\label{splitting}
f_{a_l}\circ \ldots\circ f_{a_1}(I)\cap f_{b_r}\circ\ldots\circ f_{b_1}(I)=\emptyset, 
\end{align}
where $a_l=b_r$. A sequence $a_1,\dots , a_l$ is $\pp\text{-admissible}$ if $\pp([a_1,\ldots,a_l])>0$.

\begin{theorem}
\label{thm3}
If the splitting condition is satisfied, then $\ppp(S)=1$.
\end{theorem}
Note, that the statements of Theorem \ref{thm1} and Theorem \ref{thm3} are the same as  those of Theorem 4.1 and Theorem 5.3 in \cite{DiazMatias}.  However, they are shown under different assumptions - we will prove this theorems under the assumptions that the Markov chain has countably many states (though with a very specific transition matrix), whereas in the original paper only a finite number of states was considered.

\section{Proof of Theorem 1}
In general, we will follow the proof of \cite{DiazMatias}, but with a few modifications necessary due to the countable support of the chain $\ff$. For clarity, we divide the proof into several auxiliary lemmas, which we prove later.

\begin{lemma}\label{L1}
For any probability measure $\hat\nu$ defined on $(\hat{M},\hat{\mathcal{S}})$ and any continuous and bounded function $\hat{h}:\hat{M}\to \R$ we have
\begin{align*}
\int\limits_{\hat{M}} \hat{h}\dd \hat\nu = \sum_{j=0}^\infty \int\limits_I h_j \dd\nu_j,
\end{align*}
where $h_j$ is a function from $I$ to $\R$ defined as $h_j(x)=\hat{h}(j,x)$ for $j\in\nn$.
\end{lemma}

Using this lemma we can write for any $n\in\N$
\begin{equation}\label{leftSide}
\int\limits_{\hat{M}} \hat{h}\dd T^n\hat{\mu} = \sum_{j=0}^\infty \int\limits_I h_j \dd(T^n\hat{\mu})_j.
\end{equation}
Similarly, we can decompose the integral on the right-hand side of the equation~(\ref{mainEq}) as 
\begin{equation}\label{rightSide}
\int\limits_{\hat{M}} \hat{h}\dd\hat\pi\odot\ppp \stackrel{(a)}{=}
\int\limits_\Sigma \hat{h}\circ\hat\pi\dd\ppp \stackrel{(b)}{=}
\sum_{j=0}^\infty \int\limits_{[j]}\hat{h}\circ\hat{\pi}\dd\ppp\stackrel{(c)}{=}
\sum_{j=0}^\infty \int\limits_{[j]}h_j\circ\pi\dd\ppp,
\end{equation}
where in the equality $(a)$ we used change of variables for the pushforward measure, in the equality $(b)$ the $\sigma\text{-additivity}$ of the Lebesgue integral and in the equality $(c)$ the fact, that for any $\omega\in[j]=\{\omega=(\omega_0,\omega_1,\ldots):\omega_0=j\}$ 
\begin{align*}
\hat{h}\circ\hat{\pi}(\omega)=\hat{h}(j,\pi(\omega))=h_j(\pi(\omega))=h_j\circ\pi(\omega). 
\end{align*}
%from the definition of the function $\hat{\pi}$.

The key to the proof of Theorem~\ref{thm1} is the following lemma. % which relates to the terms of the sums~(\ref{leftSide}) and~(\ref{rightSide}). 
%We will present the proof of the lemma later due to its length.
\begin{lemma}\label{L2}
If $\ppp(S)=1$, then for any continuous and bounded function $h: I\to \R$
\begin{align*}
\lim_{n\to\infty} \int\limits_I h\dd(T^n\hat{\mu})_j = \int\limits_{[j]}h\circ \pi \dd\ppp
\end{align*}
for any $j\in\nn$.
\end{lemma}

Finally, to avoid the problems with changing the order of the sum and the integral, we show in Lemma~\ref{L3} that from some term onwards both sums are already negligible.
\begin{lemma}\label{L3}
For a given continuous and bounded function $\hat{h}:\hat{M}\to \R$, probability measure $\hat{\mu}$ on $(\hat{M},\hat{\mathcal{S}})$ and any $\varepsilon>0$ there exist integers $N,M$ such that 
\small
\begin{align*}
\left|\sum_{j=M}^\infty \int\limits_{[j]}h_j\circ\pi\dd\ppp \right|&<\varepsilon,\\
\left|\sum_{j=M}^\infty \int\limits_I h_j\dd(T^n\hat\mu)_j \right|&<\varepsilon
\end{align*}
\normalsize
for any $n\ge N$.
\end{lemma}

The proof of Theorem~\ref{thm1} is then just a matter of combining these lemmas.  

\begin{proof}[Proof of Theorem~\ref{thm1}]
Using equations~(\ref{leftSide}) and~(\ref{rightSide}) it is sufficient to show that for any $\varepsilon>0$ there exists an $N\in \N$ such that 
\begin{align*}
\left|
\sum_{j=0}^\infty \int\limits_I h_j \dd (T^n\hat{\mu})_j -  
\sum_{j=0}^\infty \int\limits_{[j]}h_j\circ\pi\dd\ppp
\right|<\varepsilon
\end{align*}
for any $n\ge N$. Next, based on Lemma~\ref{L3}, we can choose integers $n_1$ and $M$ for which
\begin{align*}
\left|\sum_{j=M}^\infty \int\limits_{[j]} h_j\circ\pi \dd\ppp \right|&<\frac{\varepsilon}{3} ,\\
\left|\sum_{j=M}^\infty \int\limits_I h_j\dd(T^n\hat\mu)_j\right| &<\frac{\varepsilon}{3} 
\end{align*}
for any $n\ge n_1$. Finally, Lemma~\ref{L2} guarantees the existence of an integer $N\ge n_1$ satisfying 
\begin{align*}
\left|
\int\limits_I h_j \dd (T^n\hat{\mu})_j - \int\limits_{[j]}h_j\circ\pi\dd\ppp
\right| < \frac{\varepsilon}{3M}
\end{align*}
for any $n\ge N$ and every $j=0,1,\ldots,M-1$. Then for any $n\ge N$ we have  
\small
\begin{align*}
&\left|
\sum_{j=0}^\infty \int\limits_I h_j \dd (T^n\hat{\mu})_j -  
\sum_{j=0}^\infty \int\limits_{[j]}h_j\circ\pi\dd\ppp
\right| = \\
&=
\left|
\sum_{j=0}^{M-1} \int\limits_I h_j \dd (T^n\hat{\mu})_j -  
\sum_{j=0}^{M-1} \int\limits_{[j]}h_j\circ\pi\dd\ppp + 
\sum_{j=M}^\infty \int\limits_I h_j \dd (T^n\hat{\mu})_j -  
\sum_{j=M}^\infty \int\limits_{[j]}h_j\circ\pi\dd\ppp
\right|\le\\
&\le
\sum_{j=0}^{M-1} \left|\int\limits_I h_j \dd (T^n\hat{\mu})_j - \int\limits_{[j]}h_j\circ\pi\dd\ppp\right| + 
\left| \sum_{j=M}^\infty \int\limits_I h_j \dd (T^n\hat{\mu})_j \right| + 
\left| \sum_{j=M}^\infty \int\limits_{[j]}h_j\circ\pi\dd\ppp \right| <\\
&< 
M\frac{\varepsilon}{3M} + \frac{\varepsilon}{3} + \frac{\varepsilon}{3} =\varepsilon, 
\end{align*}
\normalsize
which concludes the proof.
\end{proof}

\subsection{Proofs of the auxiliary lemmas}
For the proof of Lemma~\ref{L1}, it suffices to use some of the properties of the Lebesgue integral. For completeness, we provide it in the Appendix. 

In the proof of Lemmas~\ref{L2} and~\ref{L3} we will also need the following lemma and its corollary. 
\newcommand{\sumy}{\sum_{\xi_n=0}^\infty\ldots\sum_{\xi_1=0}^\infty}
\newcommand{\pravd}{p_{\xi_n\xi_{n-1}}\cdot\ldots\cdot p_{\xi_2\xi_1}\cdot p_{\xi_1j}}
\newcommand{\fcie}{f_j\circ f_{\xi_1}\circ\ldots\circ f_{\xi_{n-1}}}

\begin{lemma}\label{L4}
For any $C\in\B$ and any probability measure $\hat{\mu}$ on $(\hat{M},\hat{\mathcal{S}})$
\small
\begin{align}
\label{vzL4}
(T^n\hat{\mu})_j(C)=\sumy \pravd\cdot(\fcie)\odot \mu_{\xi_n}(C).
\end{align}
\normalsize
\end{lemma}

\begin{corollary}
\label{C1}
For any probability measure $\hat{\mu}$ on $(\hat{M},\hat{\mathcal{S}})$
\begin{align*}
\lim_{n\to\infty}\sum_{j=0}^\infty \left|(T^n\hat\mu)_j(I)-m_j\right| = 0.
\end{align*}
\end{corollary}
The proofs of Lemma~\ref{L4} and Corollary~\ref{C1} are included in the Appendix.

The proof of Lemma~\ref{L2} is relatively long and technical, so for clarity we first give a brief sketch of the proof. %briefly sketch how the proof will proceed.

Given a continuous and bounded function $h:I\to \R$, a probability measure $\hat{\mu}$ on $(\hat{M},\hat{\mathcal{S}})$ and $j\in\nn$, we want to show that 
\begin{align*}
\lim_{n\to\infty} \int\limits_I h\dd(T^n\hat{\mu})_j = \int\limits_{[j]} h\circ\pi \dd\ppp. 
\end{align*}
Take $n^\star=n_1+n$ sufficiently large such that for the measure $\hat{\nu}\equiv T^{n_1}\hat\mu$ we have that the distribution $\left(\nu_i(I)\right)_{i=0}^\infty$ is sufficiently close to  $(m_i)_{i=0}^\infty$ (that can be done thanks to the Corollary~\ref{C1}). Next, using Lemma~\ref{L4} we can write
\begin{align*}
&\int\limits_I h\dd(T^{n^\star}\hat\mu)_j=\int\limits_I h\dd(T^n(T^{n_1}\hat\mu))_j =\int\limits_I h\dd(T^n\hat\nu)_j=\\
&=\sumy\pravd \cdot \int\limits_I h\dd(\fcie)\odot\nu_{\xi_n}\approx(\star). 
\end{align*}
Note that the term $\pravd$ corresponds to an $n\text{-steps}$ transition probability from the state $\xi_n$ to the state $j$ and there are only finitely many admissible paths to the state $j$ after $n\text{-steps}$. Therefore, although the expression looks like a countable sum, it is in fact a finite sum, because only a finite number of terms is non-zero. 

Next, using change of variables for the pushforward measure we can write 
\begin{align*}
\label{longInt}
\int\limits_I h\dd(\fcie)\odot\nu_{\xi_n}=\int\limits_I h\circ\fcie\dd\nu_{\xi_n}.
\end{align*}

Since for every $\xi=(\xi_0,\xi_1,\ldots)\in S$ the limit
\begin{align*}
\pi(\xi)=\lim_{n\to\infty} f_{\xi_0}\circ\ldots\circ f_{\xi_n}(x)
\end{align*}
exists regardless of $x$, the value $h\circ\fcie$ can be for sufficiently large $n$ approximated by the value $h\circ \pi(\xi^\star_{(j,\xi_1,\ldots,\xi_{n-1})})$, where $\xi^\star_{(j,\xi_1,\ldots,\xi_{n-1})}$ is an element of the set $[j,\xi_1,\ldots,\xi_{n-1}]\cap S$. Therefore we can use the approximation
\begin{align*}
\int\limits_I h\circ\fcie\dd\nu_{\xi_n}&\approx 
h\circ\pi(\xi^\star_{(j,\xi_1,\ldots,\xi_{n-1})})\int\limits_I 1 \dd\nu_{\xi_n}=\\
&=h\circ\pi(\xi^\star_{(j,\xi_1,\ldots,\xi_{n-1})})\cdot\nu_{\xi_n}(I).
\end{align*}
More precisely, the following lemma from \cite{DiazMatias} can be used to justify this approximation. 
\begin{lemma}(\cite{DiazMatias} Lemma 5.2.)
\label{L52}
For any sequence $\{\nu_n\}_{n=0}^\infty$ of probability measures on $\left(I,\mathcal{B}(I)\right)$, any $\omega\in S$ and any continuous and bounded function $h:I\to\R$
\begin{align*}
\lim_{n\to\infty} \int\limits_{I} h \dd  \left(f_{\omega_0}\circ\ldots\circ f_{\omega_n}\right)\odot \nu_n=h\circ \pi(\omega).
\end{align*}
\end{lemma}
Lemma~\ref{L52} is taken from \cite{DiazMatias}. For the completeness we present its proof in the Appendix.

Finally, $\nu_{\xi_n}(I)$ can be approximated by the value $m_{\xi_n}$, hence we can write 
\begin{align*}
(\star)&\approx \sumy m_{\xi_n}\cdot\pravd \cdot h\circ \pi(\xi^\star_{(j,\xi_1,\ldots,\xi_{n-1})})=\\&=
\sumy \ppp([j,\xi_1,\ldots,\xi_{n-1},\xi_n])\cdot h\circ \pi(\xi^\star_{(j,\xi_1,\ldots,\xi_{n-1})})\approx \int\limits_{[j]} h\circ \pi \dd\ppp.
\end{align*}
Now we will proceed to the  exact proof of Lemma~\ref{L2}.

\begin{proof}[Proof of Lemma~\ref{L2}]
Given $\varepsilon>0$ we want to show that there exists an integer $N$ such that
\begin{align*}
\left| \int\limits_{I} h\dd(T^n\hat{\mu})_j - \int\limits_{[j]}h\circ \pi \dd\ppp\right|<\varepsilon
\end{align*}
for every $n\ge N$. Since $\sum_{i=0}\limits^\infty m_i = 1$ and $||h||\equiv \max\limits_{x\in I} h(x)$ is finite, there exists an~integer~$M$ such that 
\begin{equation}\label{chvost1}
\sum_{i=M}^\infty m_i <\frac{\varepsilon}{5||h||}. 
\end{equation}
Next, using Corollary~\ref{C1}, there exists an integer $n_1$ such that 
\begin{align*}
|(T^{n_1}\hat{\mu})_i(I)-m_i|<\frac{\varepsilon^\star}{M}
\end{align*}
for every $i=0,1,\ldots,M-1$, where 
\begin{align*}
\textstyle \varepsilon^\star = \min\left(
\frac{m^\star}{2},\frac{\varepsilon}{5||h||}
\right)\quad \text{and} \quad m^\star = \min\{m_i: i=0,1,\ldots, M-1\}.
\end{align*}
The choice of $\varepsilon^\star$ ensures that for the measure $\hat{\nu}\equiv T^{n_1}\hat{\mu}$ the following hold:
\begin{itemize}
\item We have
\begin{equation}
\label{stac}
\sum_{i=0}^{M-1} \left|\nu_i(I) - m_i\right| < \varepsilon^\star \le \frac{\varepsilon}{5||h||}.
\end{equation}
\item For $i=0,1,\ldots M-1$ 
 \[\nu_i(I)>m_i-\frac{\varepsilon^\star}{M}\ge m_i-\varepsilon^\star \ge m_i-\frac{m^\star}{2}\ge m_i-\frac{m_i}{2}>0.
\] Thus, it is possible to define  probability measures $\bar\nu_i$ for $i=0,1,\ldots,M-1$ as 
\begin{align*}
\bar\nu_i(C)=\frac{\nu_i(C)}{\nu_i(I)},
\end{align*}
for every $C\in \B$. (This will be useful for the application of Lemma~\ref{L52}.) 
\item Since $\nu_i(I)>m_i-\frac{\varepsilon^\star}{M}$ for every $i=0,1,\ldots,M-1$, we obtain
\small
\begin{align*}
\sum_{i=0}^{M-1}\nu_i(I)>\sum_{i=0}^{M-1} \left(m_i-\frac{\varepsilon^\star}{M}\right)=\sum_{i=0}^{M-1}m_i-\varepsilon^\star\stackrel{(a)}{>}\left(1-\frac{\varepsilon}{5||h||}\right)-\frac{\varepsilon}{5||h||}=1-\frac{2\varepsilon}{5||h||}.
\end{align*}
\normalsize
The first term in the inequality $(a)$ follows from (\ref{chvost1}), and the second term follows from the definition of~$\varepsilon^\star$.
\item From the preceding information and the fact that $(\nu_i(I))_{i=0}^\infty$ is a probability distribution we have
\begin{equation}
\label{chvost2}
\sum_{i=M}^\infty \nu_i(I) < \frac{2\varepsilon}{5||h||}.
\end{equation}
\end{itemize}
Next, using probability measures $\bar\nu_i$ for $i=0,1,\ldots,M-1$, we consider a sequence of functions $\{\bar{G}_n\}_{n=1}^\infty$ defined for $\xi=(\xi_0,\xi_1,\ldots)$ by the formula
\begin{align*}
\bar{G}_n(\xi) = 
\begin{cases}
\int\limits_{I} h\dd(\fcie)\odot\bar\nu_{\xi_n} & \text{if }\xi_n < M, \\ 
\int\limits_{I} h\dd(\fcie)\odot\bar\nu_{\xi_{M-1}} & \text{if }\xi_n \ge M.
\end{cases} 
\end{align*}
Applying Lemma~\ref{L52}, we get that $\ppp\text{-almost everywhere}$ $\bar G_n(\xi)\to h\circ \pi(\xi)$. Moreover, $\bar G_n(\xi)\le ||h||$ for every $n\in\N_0$ due to integration with respect to the probability measure. Therefore, from Lebesgue's theorem we have
\begin{align*}
\lim_{n\to\infty} \int\limits_{[j]} \bar G_n \dd\ppp = \int_{[j]}\limits \lim_{n\to\infty} \bar G_n \dd\ppp = \int\limits_{[j]} h\circ \pi \dd\ppp.
\end{align*}
Hence there exists $n_2\in\mathbb{N}$ such that 
\begin{equation}
\label{integral}
\left|\int\limits_{[j]} \bar G_n \dd\ppp-\int\limits_{[j]} h\circ \pi \dd\ppp\right|<\frac{\varepsilon}{5}   
\end{equation}
for every $n\ge n_2$. The function $\bar G_n(\xi)$ depends only on $\xi_1,\ldots,\xi_n$, whereas $\ppp([j,\xi_1,\ldots,\xi_n])>0$ only for finitely many cases. Therefore, the function $\bar{G}_n$ can be viewed as a simple function and we can write 
\newcommand{\ssumy}{\sum_{\xi_n=0}^{M-1}\sum_{\xi_{n-1}=0}^\infty\ldots\sum_{\xi_1=0}^\infty}
\newcommand{\sssumy}{\sum_{\xi_n=M}^\infty\sum_{\xi_{n-1}=0}^\infty\ldots\sum_{\xi_1=0}^\infty}
\newcommand{\ssssumy}{\sum_{\xi_{n-1}=0}^\infty\ldots\sum_{\xi_1=0}^\infty}
\begin{align*}
\int\limits_{[j]} \bar{G}_n \dd\ppp &= \ssumy\ppp([j,\xi_1,\ldots,\xi_n]) \int\limits_{I} h \dd(\fcie)\odot\bar\nu_{\xi_n} + \\
&+ \sssumy\ppp([j,\xi_1,\ldots,\xi_n]) \int\limits_{I} h \dd(\fcie)\odot\bar\nu_{M-1}.
\end{align*}
Further for the second term, we can write 
\begin{align*}
& \sssumy  \ppp([j,\xi_1,\ldots,\xi_n])\int\limits_{I} h\dd(\fcie)\odot\bar\nu_{M-1}\stackrel{(a)}{\le}\\
&\stackrel{(a)}{\le}
||h|| \sssumy \ppp([j,\xi_1,\ldots,\xi_n]) =\\
&= 
||h|| \sssumy m_{\xi_n}\pravd =\\
&= 
||h||\sum_{\xi_n=M}^\infty m_{\xi_n}\sum_{\xi_{n-1}=0}^\infty
\ldots\sum_{\xi_1=0}^\infty \pravd\stackrel{(b)}{\le} \\
&\stackrel{(b)}{\le} ||h||\sum_{\xi_n=M}^\infty m_{\xi_n} < ||h||\frac{\varepsilon}{5||h||}=\frac{\varepsilon}{5},
\end{align*}
where in the inequality $(a)$ we used that we integrate a bounded function with respect to a probability measure and in the inequality $(b)$ we used that the sum 
\begin{align*}
\sum_{\xi_{n-1}=0}^\infty\ldots\sum_{\xi_1=0}^\infty\pravd
\end{align*}
is just an $n\text{-step}$ transition probability from  $\xi_n$ to the state $j$, therefore it can be bounded~\text{by $1$}. Putting this together with~(\ref{integral}), we have for $n\ge n_2$
\footnotesize
\begin{align}
\label{druhyEps}
\begin{split}
\left|\ssumy\ppp([j,\xi_1,\ldots,\xi_{n}])\cdot\int\limits_{I} h\dd(\fcie)\odot\bar\nu_{\xi_n}-\int\limits_{[j]} h\circ \pi\dd\ppp\right|<\frac{2\varepsilon}{5}.
\end{split}
\end{align}
\normalsize
Next, we can write
\footnotesize
\begin{align*}
&\ssumy\ppp([j,\xi_1,\ldots,\xi_{n}])\cdot\int\limits_{I} h\dd(\fcie)\odot\bar\nu_{\xi_n} = \\
&=
\ssumy m_{\xi_n}\pravd \cdot\int\limits_{I} h\dd(\fcie)\odot\bar\nu_{\xi_n}=\\
&=
\sum_{\xi_n=0}^{M-1}(m_{\xi_n}-\nu_{\xi_{n}}(I)+\nu_{\xi_n}(I))\ssssumy \pravd\cdot\int\limits_{I} h\dd(\fcie)\odot\bar\nu_{\xi_n}=\\
&=  
\sum_{\xi_n=0}^{M-1}(m_{\xi_n}-\nu_{\xi_{n}}(I))\ssssumy \pravd\cdot\int\limits_{I} h\dd(\fcie)\odot\bar\nu_{\xi_n} + \\
&+
\sum_{\xi_n=0}^{M-1}\nu_{\xi_n}(I)\ssssumy \pravd\cdot\int\limits_{I} h\dd(\fcie)\odot\bar\nu_{\xi_n}, 
\end{align*}
\normalsize
where for the first sum we have
\footnotesize
\begin{align*}
&\left|
\sum_{\xi_n=0}^{M-1}(m_{\xi_n}-\nu_{\xi_{n}}(I))\ssssumy \pravd\cdot\int\limits_{I} h\dd(\fcie)\odot\bar\nu_{\xi_n}\right| \le \\
&\le
\sum_{\xi_n=0}^{M-1}|m_{\xi_n}-\nu_{\xi_n}(I)|\ssssumy\pravd\cdot\int\limits_{I} ||h|| \dd(\fcie)\odot\bar\nu_{\xi_n}\stackrel{(a)}{\le} \\
&\stackrel{(a)}{\le} ||h||\sum_{\xi_n=0}^{M-1}\left|m_{\xi_{n}}-\nu_{\xi_n}(I)\right|\stackrel{(b)}{\le}  ||h||\frac{\varepsilon}{5||h||} = \frac{\varepsilon}{5}.
\end{align*}
\normalsize
In the inequality $(a)$ we used the same reasoning as above (integrating bounded function and the fact that the $n\text{-step}$ transition probability cannot be greater than $1$) and in the inequality $(b)$ we used~(\ref{stac}). Adding this to the inequality~(\ref{druhyEps}) we obtain for $n\ge n_2$
\footnotesize
\begin{align*}
%\label{tretiEps}
\left|
\sum_{\xi_n=0}^{M-1}\nu_{\xi_n}(I)\ssssumy \pravd\int\limits_{I} h\dd(\fcie)\odot\bar\nu_{\xi_n}-\int\limits_{[j]} h\circ \pi\dd\ppp\right|<\frac{3\varepsilon}{5}.
\end{align*}
\normalsize
Finally, using the definition of the measure $\bar\nu_{\xi_n}$ we have
\begin{align*}
&\sum_{\xi_n=0}^{M-1}\nu_{\xi_n}(I)\ssssumy \pravd\int\limits_{I} h\dd(\fcie)\odot\bar\nu_{\xi_n} =\\
&=
\ssumy\pravd\int\limits_{I} h \dd(\fcie)\odot \nu_{\xi_n}, 
\end{align*}
\normalsize
which differs from 
\begin{align*}
\int\limits_I h \dd(T^n\hat\nu)_j=\sumy\pravd  \int\limits_{I} h\dd(\fcie)\odot\nu_{\xi_n}
\end{align*}
only by a term
\begin{align*}
\sum_{\xi_n=M}^\infty \ssssumy \pravd \int\limits_{I} h\dd(\fcie)\odot\nu_{\xi_n}. 
\end{align*}
This can be bounded as
\begin{align*}
& \left|\sum_{\xi_n=M}^\infty \ssssumy \pravd \int\limits_{I} h\dd(\fcie)\odot\nu_{\xi_n}\right|\le\\
&\le 
\sum_{\xi_n=M}^\infty\ssssumy \pravd \int\limits_{I} ||h|| \dd(\fcie)\odot\nu_{\xi_n}=\\
&=||h||\sum_{\xi_n=M}^\infty \nu_{\xi_n}(I)\ssssumy\pravd\le\\
&\le ||h||\sum_{\xi_n=M}^\infty\nu_{\xi_n}(I)< ||h||\frac{2\varepsilon}{5||h||} = \frac{2\varepsilon}{5},
\end{align*}
thus 
\begin{align*}
\left| \int\limits_I h\dd(T^n\hat\nu)_j-\int\limits_{[j]} h\circ \pi \dd\ppp\right| <\varepsilon
\end{align*}
for every $n\ge n_2$. Therefore for any $n\ge n_1+n_2$ we have 
\begin{align*}
\left| \int\limits_I h\dd(T^n\hat\mu)_j-\int\limits_{[j]} h\circ \pi \dd\ppp\right| <\varepsilon
\end{align*}
which concludes the proof.
\end{proof}

\begin{proof}[Proof of Lemma~\ref{L3}]\phantom{\qedhere}\hspace{-0.85cm}
Let $\varepsilon>0$ be arbitrary. Choose $M$ such that
\begin{align*}
\sum_{j=M}^\infty m_j<\frac{\varepsilon}{2||\hat{h}||}
\end{align*}
where $m=\left(m_j\right)_{j=0}^{\infty}$ is the stationary distribution of the Markov chain $\left\{F_n\right\}_{n=0}^{\infty}$ and $||\hat{h}||\equiv \sup\limits_{x\in\hat{M}}\hat{h}(x)$. From Corollary~\ref{C1} we have  
\begin{align*}
\lim_{n\to\infty}\sum_{j=0}^\infty|(T^n\hat{\mu})_j(I) -m_j|=0,
\end{align*}
hence there is a nonnegative integer $N$ such that for any $n\ge N$
\begin{align*}
\sum_{j=0}^\infty \left| \tn(I)-m_j\right|<\frac{\varepsilon}{2||\hat{h}||}.
\end{align*}
Finally, we get that
\begin{align*}
\sum_{j=M}^\infty \int\limits_I h_j \dd\tn &\le\sum_{j=M}^\infty ||\hat{h}||\tn(I) \le ||\hat{h}||\sum_{j=M}^\infty \left|\tn(I) -  m_j+m_j\right|\leq\\
&\leq ||\hat{h}||\sum_{j=M}^\infty \left|\tn(I)- m_j\right|+||\hat{h}||\sum_{j=M}^\infty m_j<\varepsilon.
\end{align*}
Similarly
\begin{align*}
\left|\sum_{j=M}^\infty \int\limits_{[j]} h_j\circ \pi \dd \ppp\right|\le \left|\sum_{j=M}^\infty \int\limits_{[j]} ||\hat{h}|| \dd \ppp\right|\le ||\hat{h}||\sum_{j=M}^\infty m_j <\varepsilon. \tag*{\qed}
\end{align*}
\end{proof}

\section{Proof of Theorem 2}

Recall that if we consider Markov chain $\left\{F_n\right\}_{n=0}^\infty$ on the space $\left(\Sigma,\F,\ppp\right)$, where $F_n(\xi)=\xi_n$, then the stationary distribution of this chain is $m=(m_i)_{i=0}^\infty$, where $m_i=\frac{1}{1+E}\sum_{j=i}^\infty p_i$. Thus, according to the law of large numbers in Markov chains (see e.g. \cite{BM} Theorem 10.1 on page 185) if $\int_{\nn} |h|\dd m <\infty$ for some function $h: \nn\to\R$, then $\ppp\text{-almost}$ surely 
\begin{align}
\label{silny zakon}
\lim\limits_{n\to\infty}\frac{1}{n}\sum_{k=0}^{n-1}h(F_k) = \int\limits_{\nn} h\dd m. 
\end{align}
If we take the function $h(i)=\log L_i$, then on the right side of the equation (\ref{silny zakon}) we get 
\begin{align*}
\int\limits_{\nn} \log F_i\dd m = \sum_{i=0}^\infty (\log F_i)\cdot m_i = \sum\limits_{i=0}^\infty (\log F_i)\cdot \ppp(F_0=i) = E_{\ppp}(\log L_{F_0})<0. 
\end{align*}
It follows that for the set 
\begin{align*}
S^+\equiv\left\{
\xi\in\Sigma: \lim_{n\to\infty} \frac{1}{n} \left(\log L_{\xi_0} + \log L_{\xi_1}+\ldots + \log L_{\xi_{n-1}}\right) = E_{\ppp}(\log L_{F_0})
\right\}
\end{align*}
we have $\ppp(S^+)=1$, hence to prove the theorem it is sufficient to show that $S^+\subseteq~S$. Take any $\xi\in S^+$ and denote $L_{f_{\xi_0}\circ \ldots \circ f_{\xi_{n-1}}}$ the Lipschitz constant of the function $f_{\xi_0}\circ~\ldots\circ f_{\xi_{n-1}}$. The definition of the Lipschitz function implies that
\begin{align*}
\frac{\di(I_\xi^n)}{\di(I)} \le L_{f_{\xi_0}\circ \ldots \circ f_{\xi_{n-1}}} \le L_{\xi_0}\cdot \ldots \cdot L_{\xi_{n-1}}, 
\end{align*}
therefore 
\begin{align*}
&\limsup_{n\to\infty}\frac{1}{n}\log \di(I_\xi^n)\le 
\limsup_{n\to\infty}\frac{1}{n}\log \left(L_{\xi_0}\cdot \ldots \cdot L_{\xi_{n-1}}\cdot \di(I)\right) =\\
&=
\limsup_{n\to\infty}\frac{1}{n}\log \di(I) + \limsup_{n\to\infty}\frac{1}{n}\left(\log L_{\xi_0} +\ldots +\log L_{\xi_{n-1}}\right) = E_{\ppp}(\log L_{F_0}) < 0
\end{align*}
(the first term is clearly $0$ and the second term follows from the definition of $S^+$). Thus, $\lim\limits_{n\to\infty} \di(I_\xi^n)=0$ (if this limit were positive, then $\frac{1}{n}\log \di(I_{\xi}^n)\to 0$ for $n\to\infty$ which is not true). Note that this limit exists, because the sequence $\{\di(I_\xi^n)\}_{n=1}^\infty$ is nonincreasing and bounded. Since $I_\xi\subseteq I_{\xi}^n$ for every $n\in\N$, we obtain 
\begin{align*}
\di(I_\xi)\le \lim_{n\to\infty} \di(I_{\xi}^n) = 0, 
\end{align*}
hence $\xi\in S$ and $S^+\subseteq S$, which concludes the proof. \hfill \qedsymbol{}

\section{Proof of Theorem 3}
Similarly to the proof of Theorem 1, we will follow the ideas of the proof in~\cite{DiazMatias} with some adjustments needed due to different assumptions.

For any $x\in I$ and $n\in\N$ consider the sets
\begin{align*}
S_n^x&\equiv \{\xi\in\Sigma: x\in I_\xi^n\},\\
\quad S^x&\equiv \{\xi\in\Sigma: x\in I_\xi\}.
\end{align*}
Since $I_\xi^{n+1}\subseteq I_{\xi}$ and $I_\xi=\bigcap\limits_{n=1}^\infty I_\xi^n$, it clearly follows that
\begin{itemize}
\item $S_{n+1}^x\subseteq S_n^x$,
\item $S^x=\bigcap\limits_{n=1}^\infty S_n^x$, hence $\ppp(S_n^x)\to \ppp(S^x)$ for $n\to\infty$. 
\end{itemize}
If $\xi\notin S$, then $I_\xi$ is a closed interval, hence it contains some rational number. Therefore there exists  $q\in \Q\cap I$ such that $\xi\in S^q$. Consequently,
\begin{align*}
\Sigma\setminus S \subseteq \bigcup_{q\in \Q\cap I} S^q
\end{align*}
and 
\begin{align*}
1-\ppp(S)\le \sum_{q\in \Q\cap I} \ppp(S^q). 
\end{align*}
To prove the theorem it is sufficient to show that $\ppp(S^x)=0$ for any $x\in I$ or alternatively  $\ppp(S_n^x)\to~0$ for $n\to\infty$. In addition, the sequence $\{\ppp(S_n^n)\}_{n=1}^\infty$ is nonincreasing, hence it suffices to prove that $\ppp(S_{nN}^x)\to 0$ for $n\to\infty$ for some positive integer $N$. We show this using the following two lemmas. 
\begin{lemma}\label{splittingInequality}
Denote $\sigma^N$ the $N\text{-th}$ iterate of the shift map, i.e. for $\omega=(\omega_0,\omega_1,\ldots)\in~\Sigma$, $\sigma^N(\omega)=(\omega_N,\omega_{N+1},\ldots)$.
If the splitting condition is satisfied, then there exists a positive integer $N$ and a $\ppp$-admissible cylinder $W\equiv [\xi_0,\ldots, \xi_{N-1}]$ such that for the set 
\begin{align*}
\Sigma_n^W \equiv \{\omega\in \Sigma: \sigma^{iN}(\omega)\notin W \text{ for } i=0,1,\ldots,n-1\}
\end{align*}
we have 
\begin{align*}
\ppp\left(S_{nN}^x\right)\le \ppp\left(\Sigma_n^W\right)
\end{align*}
for any positive integer $n$ and any $x\in I$. 
\end{lemma}

\begin{lemma}\label{ergodicThm}
For any positive integer $N$ and any $\ppp\text{-admissible}$ cylinder $W=[\xi_0,\ldots,\xi_{N-1}]$
\begin{align*}
\lim\limits_{n\to\infty} \ppp\left(\Sigma_n^W\right) = 0.
\end{align*}
\end{lemma}
Using these lemmas, we obtain that for any $x\in I$
\begin{align*}
\ppp(S^x) =\lim_{n\to\infty} \ppp\left(S_{nN}^x\right)\le \lim_{n\to\infty} \ppp\left(\Sigma_n^W\right)= 0
\end{align*}
what had to be proven. 

Now we prove the above lemmas.
\begin{proof}[Proof of Lemma~\ref{splittingInequality}]
The proof will be in two parts. In the first one, we will show that if the splitting condition is satisfied, then there are $\ppp\text{-admissible}$ cylinders $[\xi_0,\ldots,\xi_{N-1}]$ and $[\eta_0,\ldots,\eta_{N-1}]$ satisfying
\begin{itemize}
\item [(i)] $\xi_0=\eta_0$,
\item [(ii)] $\xi_{N-1}=\eta_{N-1}$,
\item [(iii)] $\ppp([\xi_0,\ldots,\xi_{N-1}])\le \ppp([\eta_0,\ldots,\eta_{N-1}])$,
\item [(iv)] $f_{\xi_0}\circ\ldots\circ f_{\xi_{N-1}}(I)\cap f_{\eta_0}\circ\ldots\circ f_{\eta_{N-1}}(I)=\emptyset$.
\end{itemize}
This result is the same as Claim 4.4 from \cite{DiazMatias}, but in our case there are slightly different assumptions. After showing this in the second part of the proof we can directly apply Proposition 3.1 from \cite{DiazMatias} which holds in our situation as well - for completeness, we formulate and prove it in the Appendix as Lemma~\ref{DMprop31}. It states that if $W=[\xi_0,\ldots,\xi_{N-1}]$ and $[\eta_0,\ldots,\eta_{N-1}]$ are $\ppp$-admissible cylinders satisfying (i)-(iv), then $\ppp(S_{nN}^x)\le \ppp(\Sigma_n^W)$ for any positive integer $n$ and any $x\in I$, which is exactly what we need to prove. 

In the first part of the proof we introduce two notations - for two cylinders $C=[c_1,\ldots,c_{n_1}]$ and $D=[d_1,\ldots,d_{n_2}]$ let $C\star D$ denote the cylinder 
\begin{align*}
C\star D = [c_1,\ldots,c_{n_1},d_1,\ldots,d_{n_2}]
\end{align*}
and let $f_{\circ C}$ denote the composition
\begin{align*}
f_{\circ C}=f_{c_{n_1}}\circ \ldots\circ f_{c_1}.
\end{align*}
Since the splitting condition is satisfied there are $\pp\text{-admissible}$ cylinders $A\equiv[a_1,\ldots,a_l]$ and $B\equiv[b_1,\ldots,b_r]$, where $a_l=b_r$ and 
\begin{align}
\label{splitting1}
f_{\circ A}\left(I\right)\cap f_{\circ B}(I)=\emptyset.
\end{align}
Without loss of generality let $a_1\leq b_1$. Recall  the assumption that $p_0\in\left(0,1\right)$ and the set $\left\{j\in\mathbb{N}_0: p_j>0\right\}$ is not bounded. Let $\varphi=\min\left\{n\in\mathbb{N}_0:p_{b_1+n}>0\right\}$. Then the cylinder 
\begin{align}
\label{cylinderb}
B_1\star B\equiv[0,b_1+\varphi,\ldots,b_1+1]\star [b_1,\ldots,b_r]=[0,b_1+\varphi,\ldots,b_1+1,b_1,\ldots,b_r]
\end{align}
is $\pp\text{-admissible}$ since $b_1+i$ is accessible from $b_1+(i+1)$ for any $i\in\mathbb{N}_0$ and $b_1+\varphi$ is accessible from $0$ because $p_{b_1+\varphi}>0$.

Since $a_1\leq b_1$ there is $k\in\mathbb{N}_0$ such that $a_1+k=b_1$. Then the cylinder  
\begin{align}
\label{cylindera}
\begin{split}
B_1\star A_1\star A&\equiv[0,b_{1}+\varphi,\ldots,b_{1}+1]\star[a_{1}+k,\ldots, a_{1}+2, a_{1}+1]\star[a_1,\ldots,a_l]=\\
&=[0,b_{1}+\varphi,\ldots,b_{1}+1,a_{1}+k,\ldots, a_{1}+2, a_{1}+1,a_1,\ldots,a_l]
\end{split}
\end{align}
is $\pp\text{-admissible}$ because  $a_1+i$ can be reached from $a_1+(i+1)$, $b_1+i$ can be reached from $b_1+(i+1)$ for any $i\in\mathbb{N}_0$. Further, $b_1=a_1+k$, hence $a_1+k$ is accessible from $b_1+1$ and $b_1+\varphi$ is accessible from $0$  because $p_{b_1+\varphi}>0$.

The length of the cylinder (\ref{cylinderb}) is $r+\varphi+1$ and the length of the cylinder (\ref{cylindera}) is $l+k+\varphi+1$. Firstly, we assume that the length of the cylinder (\ref{cylinderb}) $r+\varphi+1\leq l+k+\varphi+1$ where $l+k+\varphi+1$ is the length of the cylinder (\ref{cylindera}), that is $r\leq l+k$. Hence the cylinder (\ref{cylinderb}) is shorter by $l+k-r$ places. Adding $l+k-r$  zeroes to (\ref{cylinderb}) we get the cylinder 
\begin{align}
\label{cylinderc}
\begin{split}
B_0\star B_1\star B&\equiv [0,\dots,0]\star [0,b_1+\varphi,\ldots,b_1+1]\star [b_1,\ldots,b_r]=\\
&=[0,\ldots,0,b_1+\varphi,\ldots,b_1+1,b_1,\ldots,b_r].
\end{split}
\end{align}
If the cylinder (\ref{cylinderb}) is longer than the cylinder (\ref{cylindera}), then we add $r-l-k$ zeroes to (\ref{cylindera}). The remaining steps will proceed analogously.

Now we obtain that the lengths of the cylinders (\ref{cylinderc}) and (\ref{cylindera}) are the same. From~(\ref{splitting1}) we can see that for  $x,y\in I$ we have  $f_{\circ A}(x)\neq f_{\circ B}(y)$. Consequently,
\begin{align}
\label{splitting2}
f_{\circ A}\circ f_{\circ (B_1\star A_1)}\left(I\right)\cap f_{\circ B}\circ f_{\circ (B_0\star B_1)}\left(I\right)=\emptyset.
\end{align}
Since the cylinder (\ref{cylindera}) is $\pp\text{-admissible}$, we obtain
that
\begin{align*}
[\xi_0,\ldots,\xi_{N-1}]\equiv \left[a_l, \ldots, a_1,a_1+1,a_1+2,\ldots,a_1+k,b_{1}+1,\ldots,b_{1}+\varphi,0\right]
\end{align*}
is $\ppp\text{-admissible}$ and similarly the cylinder
\begin{align*}
\left[\eta_0,\ldots,\eta_{N-1}\right]\equiv \left[b_r,\ldots,b_1,b_1+1,\ldots,b_1+\varphi,0,0,\ldots,0\right],
\end{align*}
is $\ppp\text{-admissible}$. Moreover, $\xi_0=a_l=b_r=\eta_0$ and $\xi_{N-1}=0=\eta_{N-1}$, therefore the conditions (i) and (ii) are satisfied. Next,  the equation~(\ref{splitting2}) implies that $[\xi_0,\ldots,\xi_{N-1}]$ and  $[\eta_0,\ldots,\eta_{N-1}]$ satisfy the condition (iv). Finally, if they do not satisfy the third condition it is sufficient to interchange $\xi$ and $\eta$.

Applying Lemma~\ref{DMprop31} from Appendix, the proof of Lemma~\ref{splittingInequality} is complete.
\end{proof}

\begin{proof}[Proof of Lemma~\ref{ergodicThm}]
Since $\Sigma_{n+1}^W\subseteq \Sigma_n^W$, for the set
\begin{align*}
\Sigma^W\equiv \bigcap_{n=1}^\infty \Sigma_n^W
\end{align*}
we have $\ppp(\Sigma^W)=\lim\limits_{n\to\infty} \ppp(\Sigma_n^W)$. Thus it is sufficient to show that $\ppp(\Sigma^W)=0$. This can be done by showing that $\ppp(\Sigma^W)<\varepsilon$ for any $\varepsilon>0$. 

Since $W$ is a $\ppp\text{-admissible}$ cylinder, for given $\varepsilon>0$ we can take an integer $n^\star$ such that $(1-\ppp(W))^{n^\star}<\varepsilon$. Next, consider the stopping times 
\begin{align*}
\tau_1(\omega)&\equiv \min \{n\in\{0,N,2N,3N,\ldots\}: F_n(\omega)=\xi_0\},\\
\tau_k(\omega)&\equiv \min \{
n\in\{0,N,2N,3N,\ldots\}, n>\tau_{k-1}: F_n(\omega) = \xi_0
\} 
\end{align*}
for $k=2,3,\ldots, n^\star$ and the set
\begin{align*}
\Sigma^+\equiv \{
\omega\in \Sigma: \tau_1(\omega)<\infty, \ldots, \tau_{n^\star}(\omega)<\infty
\}.
\end{align*}
For this set we have $\ppp(\Sigma^+)=1$ - this can be seen using a stochastic process $\{\tilde{F}_n\}_{n=0}^\infty$ defined on the probability space $(\Sigma,\F,\ppp)$ as $\tilde{F}_n(\omega)\equiv F_{nN}(\omega)=\omega_{nN}$ for $n\in \nn$. This is clearly a Markov chain with the state space $\nn$ (since $\{F_n\}_{n=0}^\infty$ is aperiodic), transition matrix $Q^N$ and the invariant distribution $m$. Therefore all states are positive recurrent (see e.g. \cite{BM} Theorem 8.1 on page 162), hence the chain $\{\tilde{F}_{n}\}_{n=0}^\infty$ $\ppp$ almost surely visits the state $\xi_0$ infinitely many times, thus $\tau_1,\ldots,\tau_{n^\star}$ are $\ppp$ almost surely finite. 

Finally, consider sets $A_1,\ldots, A_{n^\star}$ defined as
\begin{equation}
A_n=\{\omega\in \Sigma^+: F_{\tau_n(\omega)}(\omega)=\xi_0, F_{\tau_n(\omega)+1}(\omega)=\xi_1,\ldots, F_{\tau_n(\omega)+N-1}(\omega)=\xi_{N-1}\},
\end{equation}
$n=1,\ldots,n^\star$. Using the strong Markov property, the sets $A_1,\ldots,A_{n^\star}$ are independent and $\ppp(A_n)=\ppp(W)$ for every $n=1,\ldots,n^\star$. Moreover if $\omega\in \Sigma^W$, then $\omega$ cannot be in any of the sets $A_1,\ldots,A_{n^\star}$, hence 
\begin{equation}
\ppp(\Sigma^W)\le \ppp(A_1^C\cap \ldots\cap A_{n^\star}^C)=
(1-\ppp(W))^{n^\star}<\varepsilon, 
\end{equation}
which concludes the proof. 
\end{proof}

\section{Final remarks and examples}

In the following we will show situations where the average contraction or splitting condition can be applied. The systems with impulses will be described using properties of the  functions $f$ and $g$ generating the system and the expected value of the time between impulses. Recall that in the whole paper we take assumptions which are mentioned below the definition~(\ref{MPrechodu}). We also show that for some systems both conditions are appliable, although there are systems where the application of one property is impossible while the other one may be applied. 

In the following examples  and claims we consider Lipschitz  functions $f$ and $g$ with Lipchitz constants $L_1$ and $L_0$, respectively. We use  the relation (\ref{claim1}) to find a condition for the expected value of the time between impulses so Theorem \ref{thm2} holds and the system is stable.

%\begin{exmp}
%Let $f,g:\left<0,1\right>\to \left<0,1\right> $ be defined  by $g(x)=\frac{x}{2}+\frac{1}{2}$ and $f(x)=\begin{cases}2x & \text{ if } x \in \left<0,\frac{1}{2}\right>\\ 2-2x & \text{ if } x \in \left(\frac{1}{2},1\right> \end{cases}$.
%The Lipschitz constant $L_0$ of the function $g$ is $\frac{1}{2}$, and the Lipschitz constant $L_1$ of the function $f$ is $2$. If $2\cdot\left(\frac{1}{2}\right)^E=2^{-E+1}<1$ then the condition from (\ref{claim1}) is satisfied. It means that if we take any random distribution of times between impulses with the expected value $E$ greater than $1$ then the corresponding system with impulses satisfies Theorem \ref{thm2}. 
%\end{exmp}

It is easy to see that for any system generated by two different linear maps the contraction property is satisfied.  

\begin{claim}
If $f$ and $g$ are two different linear maps then the system with impulses satisfies the inequality (\ref{claim1}).
\end{claim}

\begin{proof}
For two different linear maps we have $L_0,L_1\in \left<0,1\right>$ and $L_0<1$ or $L_1<1$, hence the inequality (\ref{claim1}) holds for any $E>0$.
\end{proof}

In the next example we find a condition on the expected value of times between impulses so that the system is stable.

\begin{example}
Let $f,g:\left<0,1\right>\to \left<0,1\right> $ be defined by $f(x)=4x(1-x)$ and $g(x)=\frac{x}{8}$. The Lipschitz constant $L_0$ of the function $g$ is $\frac{1}{8}$, and the Lipschitz constant $L_1$ of the function $f$ is $8$. It is easy to see that the condition (\ref{claim1}) is satisfied for any random distribution of impulses with the expected value $E$ greater than $1$.
\end{example}

In general, for $L_0<1$ the following result holds.

\begin{claim}
If $L_0<1$ then there exists a distribution of times between impulses such that the system with impulses satisfies the relation (\ref{claim1}).
\end{claim}

In the following example  we give the condition for the expected value of the distribution of impulses so that (\ref{claim1}) holds in a case that $L_0>1$.

\begin{example}
Let the functions $f,g:\left<0,1\right>\to \left<0,1\right> $ be defined  as $f(x)=\frac{x}{8}$ and $g(x)=4x(1-x)$. The Lipschitz constant $L_0$ of the function $g$ is $8$, and the Lipschitz constant $L_1$ of the function $f$ is $\frac{1}{8}$. The inequality (\ref{claim1}) is satisfied   for any distribution of random times with expected value $E$ less than $1$. 
\end{example}

In general, for $L_0>1$ the following claim holds.

\begin{claim}
\label{claim4}
If $L_0>1$ then there is a distribution of times between impulses such that~(\ref{claim1}) holds if and only if $\log_{L_0}L_1<0$.                          
\end{claim}

\begin{proof}
The condition from (\ref{claim1}) is satisfied if $L_1L_0^E<1$ which is equivalent to $E<\log_{L_0}\frac{1}{L_1}=-\log_{L_0}L_1$. Since $E>0$ the inequality $\log_{L_0}L_1<0$ guarantees the existence of a distribution of random times between impulses such that the system with impulses satisfies the contraction condition.
\end{proof}

\begin{remark}
If $L_0=1$ then the contraction property holds only if $L_1<1$. In this case the distribution of random times can be arbitrary.
\end{remark}

The inequality (\ref{claim1}) is appliable only for Lipchitz maps. The splitting condition does not require the map to be Lipschitz. We give a simple condition which ensures that the splitting property is satisfied.

\begin{claim}
\label{claim6}
If the functions $f,g$ are continuous, monotone and have unique attracting fixed points $a,b$, respectively, such that $g(a) \neq b$, then the splitting condition is satisfied.
\end{claim}

\begin{proof}
From the continuity of $g$ and the assumption that $g(a)\neq b$ it is possible to choose disjoint neighbourhoods  $O(a)$ of $a$ and $O(b)$ of $b$ such that $g(O(a))\cap O(b)=\emptyset $.  Since  $a \in I$ is the unique attracting fixed point of the function $f$,  for any $x_0 \in I$ the sequence $\left\{f^n\left(x_0\right)\right\}_{n=0}^{\infty}$ converges to $a$. Similiarly, $b \in I$ is the unique attracting fixed point of the function $g$ so for any $y_0 \in I$, the sequence $\left\{g^n\left(y_0\right)\right\}_{n=0}^{\infty}$ converges to $b$. Hence for the neighbourhoods $O(a)$ and $O(b)$  there exists $n\in\mathbb{N}$ so that $f^{n}\left(I\right)\subset O(a)$ and $g^{n}\left(I\right)\subset O(b)$.  Now take $b_1=n, b_2=n-1,\dots , b_n=1, b_{n+1}=0$ and $a_1=a_2=\dots =a_n=0$. The sequences $a_1, \ldots , a_n$ and $b_1, \dots , b_{n+1} $ are $\ppp\text{-admissible}$, $a_n=b_{n+1}$. Since
\begin{align*}
f_{a_n}\circ\ldots \circ f_{a_1}(I)&=g^n(I)\subset O(b),\\
f_{b_{n+1}}\circ\ldots \circ f_{b_1}(I)&=g(f^n(I))\subset g(O(a)),
\end{align*}
the splitting condition is satisfied.
%From the injectivity of the function $g$ we have that there exist a neighbourhood $O_1(a)$ of the point $a$ such that $g\left(c\right)\neq b$ for all $c\in O_1(a)$. From the continuity of the function $g$ we get that there exist a neighbourhood $O_2(a)\subseteq O_1(a)$ such that $g\left(O_2(a)\right)\cap O(b)=\emptyset$. Moreover from the continuity ot the function $g$ we have $g\left(f^{n_1}(I)\right)\subset g\left(O(a)\right)$. At the beginning, we can choose the neighborhood $O(a)$ to be a subset neighborhood $O_2(a)$. Again from the continuity of the function $g$ we have $g\left(O(a)\right)\subseteq g\left(O_2(a)\right)$. This implies that from $g\left(O_2(a)\right)\cap O(b)=\emptyset$ we get $g\left(O(a)\right)\cap O(b)=\emptyset$. Therefore we obtain that $g(f^{n_1}\left(I\right))\cap g^{n_2}\left(I\right)=\emptyset$. It means that the splitting condition is satisfied.
\end{proof}

\begin{example}
Let  $f,g:\left<0,2\right>\to \left<0,2\right> $ be defined as $f(x)=\left(1-\frac{\sqrt{2}}{2}\right)x+\sqrt{2}$ and $g(x)=\sqrt{x}$ (see Figure \ref{obrazok90}). The unique attracting fixed point of the function $f$ is $2$ and the unique attracting fixed point of $g$ is $0$. Hence by Claim~\ref{claim6} the system satisfies the splitting condition. Since the function $f$ is not Lipschitz, the contraction condition cannot be applied in this case.
%The fixed point $2$ is an attracting because $\left|f'(x)\bigg|_{x=2}\right|=\left|1-\frac{\sqrt{2}}{2}\right|<1$. The fixed point of the function $g$ is $0$. The fixed point $0$ is also attracting because $g^n\left(\left<0,2\right>\right)=\left<0,\sqrt[2n]{2}\right>$, which for $n$ approaching infinity converges to only one point $0$. It means that Claim \ref{claim6} holds and the splitting condition is satisfied.
\begin{figure}[h!]
\centering
\resizebox*{6.1cm}{!}{\includegraphics{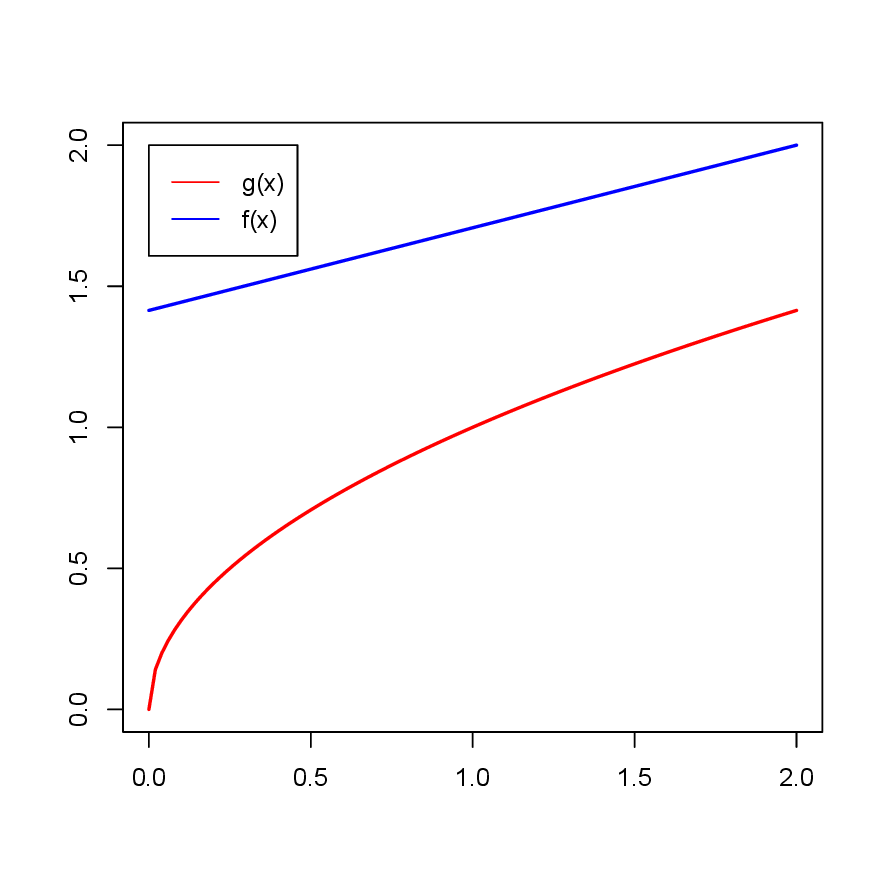}}
\caption{The functions $f(x)=\left(1-\frac{\sqrt{2}}{2}\right)x+\sqrt{2}$ (blue) and $g(x)=\sqrt{x}$ (red) on the interval~$\left<0,2\right>$.} 
\label{obrazok90}
\end{figure}
\end{example}

Random dynamical systems with impulses are related to random iterated function systems mentioned in the Introduction. Random iterated fuction system $\{\bar{X}_n\}_{n=0}^{\infty}$ generated by functions $f,g$ and parameter $p$ is defined as 
\begin{align}
\label{simple RDS}
\bar{X}_{n+1}=\begin{cases}
f(\bar{X}_n) & \text{with probability } p,\\
g(\bar{X}_n) & \text{with probability } 1-p,
\end{cases}
\end{align}
where $n\in\mathbb{N}_0$, $p\in\left(0,1\right)$ and $\bar{X}_0$ is a given starting point or a random variable.
%Let's denote $\chi_{h}\left(\xi\right)$ an indicator function, i. e. $\chi_{h}\left(\xi\right)=1$ if $\xi=h$ and $\chi_{h}\left(\xi\right)=0$ if $\xi\neq h$.

\begin{claim}
Let $\left\{X_n\right\}_{n=0}^{\infty}$ be the stochastic process defined in (\ref{split}) with the times between impulses following geometric distribution with parameter $1-p$, i.e. $p_k=(1-p)p^k$. Then $\left\{X_n\right\}_{n=0}^{\infty}$ and $\{\bar{X}_n\}_{n=0}^{\infty}$ define the same process. 
%This process has the Markov property.
%$T$ from (\ref{pulse distribution}), which has the geometric distribution and with initial distribution from (\ref{pociatocne}).
\end{claim}
\begin{proof}
We are going to show that for each $B\in\mathcal{B}\left(I\right)$ and $n\in\mathbb{N}_0$ the probability that $\bar{X}_n\in B$ is the same as the probability that $X_n\in B$.
These probabilities are determined by a sequence of functions $f,g$ which are applied to time $n$. 

Consider the probability space $\left(\Sigma,\mathcal{F},\pp^{\star}\right)$ and  the random variables $F^{\star}_n:\Sigma \to \left\{0,1\right\}$ defined by
\begin{align*}
F^{\star}_n\left(\omega\right)=\begin{cases}
0 & \text{if } \omega_n=0,\\
1 & \text{if } \omega_n\geq 1.
\end{cases}
\end{align*}
The random process $\left\{X_n\right\}_{n=0}^{\infty}$ defined in (\ref{split}) can be written as
\begin{align}
\label{proces1}
X_{n+1}\left(\omega\right)=f_{F_n^{\star}\left(\omega\right)}\left(X_n\left(\omega\right)\right),
\end{align}
for every $n\in\mathbb{N}_0$, $f_0\equiv g$ and $f_1\equiv f$. Then the probability that $X_n\in B$ is 
\begin{align}
\label{pp_geo}
\pp^\star \left(X_n\in B\right)=\sum\limits_{j_0,\ldots,j_{n-1}\in\left\{0,1\right\}: f_{j_{n-1}}\circ \ldots \circ f_{j_0}(X_0)\in B}\pp^{\star}\left(F^\star_0=j_0,\ldots,F^\star_{n-1}=j_{n-1}\right).
\end{align}
Note that the event $\{F^\star_0=j_0,\ldots,F^\star_{n-1}=j_{n-1}\}$ corresponds to some cylinder(s) on $\Sigma$, which can be used to compute the probabilities on the right-hand side of (\ref{pp_geo}). For example, the event $\{F^\star_0=1,F^\star_{1}=1,F^\star_{2}=0, F^\star_{3}=1, F^\star_{4}=0\}$ corresponds to the cylinder $[2,1,0,1,0]$, hence $\pp^\star\left(\{F^\star_0=1,F^\star_{1}=1,F^\star_{2}=0, F^\star_{3}=1, F^\star_{4}=0\}\right)=\pp^*([2,1,0,1,0])=p_2p_1=(1-p)^2p^3$. To construct the corresponding cylinder let $K$ express the number of $g\text{'s}$, i.e.  $K=n-\sum\limits_{i=0}^{n-1}j_i$. For~$K>0$ 
 denote 
\begin{align*}
\vartheta_1=\min\{0\leq k\leq n-1:j_k=0\},\\
\vartheta_l=\min\{\vartheta_{l-1}< k\leq n-1:j_k=0\},
\end{align*}
and cylinders
\begin{align*}
\Xi_1=\left[\vartheta_1,\vartheta_1-1,\vartheta_1-2,\ldots,0\right],\\
\Xi_l=\left[\vartheta_l-\vartheta_{l-1}-1, \vartheta_l-\vartheta_{l-1}-2,\ldots,0\right],
\end{align*}
for $l=2,\ldots,K$. If $j_{n-1}=0$ the term from the sum in (\ref{pp_geo}) can be expressed  in the following way
\begin{align*}
&\pp^{\star}\left(F^\star_0=j_0,\ldots,F^\star_{n-1}=j_{n-1}\right)=\pp^{\star}\left(\Xi_1\star \Xi_2\star \ldots \star \Xi_{K}\right)=p_{\vartheta_1}\prod\limits_{l=2}^{K}p_{\vartheta_{l}-\vartheta_{l-1}-1}=\\
&=p^{\vartheta_1}(1-p)\prod\limits_{l=2}^{K}p^{\vartheta_{l}-\vartheta_{l-1}-1}(1-p)=(1-p)^K p^{\vartheta_1}\prod\limits_{l=2}^{K}p^{\vartheta_{l}-\vartheta_{l-1}-1}=\\
&=(1-p)^Kp^{\vartheta_K-(K-1)}=|\text{ since $j_{n-1}=0$ }|=\\
&=(1-p)^Kp^{n-1-(K-1)}=(1-p)^Kp^{n-K}.
\end{align*}
If $j_{n-1}=1$, the situation is more complicated, because there are more cylinders corresponding to the event $\{F^\star_0=j_0,\ldots,F^\star_{n-1}=j_{n-1}\}$ - for example, the event $\{F^\star_0=1,F^\star_{1}=0,F^\star_{2}=1, F^\star_{3}=1 \}$ corresponds to the union of cylinders $[1,0,2,1], [1,0,3,2], [1,0,4,3],\ldots$. Formally, if $j_{n-1}=1$, then $n-1-\vartheta_K>0$ and we can denote cylinders
\begin{align*}
\Xi^{d}_{K+1}=\left[d, d-1,\ldots,d-n+\vartheta_K+2\right],
\end{align*}
where $d\in\mathbb{N}_0$ such that $d\geq n-1-\vartheta_K$ (in the above-mentioned example $n=4$, $K=1$ and $\vartheta_K=1$ and the corresponding cylinders are $[1,0]\star[2,1], [1,0]\star[3,2],[1,0]\star[4,3]\ldots$, i.e. the ,,ending'' cylinders are $\Xi_2^2=[2,1]$, $\Xi_2^3=[3,2]$, $\Xi_2^4=[4,3]$, etc.).
In this case the term from the sum in (\ref{pp_geo}) can be expressed
\begin{align*}
&\pp^{\star}\left(F^\star_0=j_0,\ldots,F^\star_{n-1}=j_{n-1}\right)=\sum\limits_{d=n-1-\vartheta_K}^{\infty}\pp^{\star}\left(\Xi_1\star \Xi_2\star \ldots \star \Xi_{K}\star \Xi^{d}_{K+1} \right)=\\
&=\left(p_{\vartheta_1}\prod\limits_{l=2}^{K}p_{\vartheta_{l}-\vartheta_{l-1}-1}\right)\sum\limits_{d=n-1-\vartheta_K}^{\infty}p_{d}=\\
&=\left(p^{\vartheta_1}(1-p)\prod\limits_{l=2}^{K}p^{\vartheta_{l}-\vartheta_{l-1}-1}(1-p)\right)\sum\limits_{d=n-1-\vartheta_K}^{\infty}p^d(1-p)=\\
&=\left(p^{\vartheta_1}(1-p)\prod\limits_{l=2}^{K}p^{\vartheta_{l}-\vartheta_{l-1}-1}(1-p)\right)(1-p)p^{n-1-\vartheta_K}\sum\limits_{d=0}^{\infty}p^d=\\
&=p^{n-K}(1-p)^{K+1}\frac{1}{1-p}=(1-p)^{K}p^{n-K}.
\end{align*}

If $K=0$, then  $j_i=1$ for $i=0,\ldots,n-1$ and 
\begin{align*}
&\pp^{\star}\left(F^\star_0=j_0,\ldots,F^\star_{n-1}=j_{n-1}\right)=\sum\limits_{d=n}^{\infty}\pp^{\star}\left(\left[d\right]\right)=\sum\limits_{d=n}^{\infty}p^d(1-p)=\\
&=p^{n}(1-p)\sum\limits_{d=0}^{\infty}p^d=p^{n}(1-p)\frac{1}{1-p}=p^{n}(1-p)^0=p^{n-K}(1-p)^K.
\end{align*}

In each situation described above, we have obtained
\begin{align*}
\pp^{\star}\left(F^\star_0=j_0,\ldots,F^\star_{n-1}=j_{n-1}\right)=p^{n-K}(1-p)^K,
\end{align*}
which is exactly the probability corresponding to (\ref{simple RDS}), i.e.  to  the random iterated function system  generated by functions $f, g$, where the functions are chosen independently with probability $p$ and $1-p$, repectively. 
\end{proof}
%The following claim gives a relation between the stationary distribution of the random process $\{Z_n\}_{n=0}^\infty$ defined in $1.1$ and the stationary distribution of the corresponding random iteration function system in case that this stationary distribution is unique.

Let $\{\tilde{X}_n\}_{n=0}^{\infty}$ be a random iterated function system with probabilities, where $\tilde{f}_k:I\to I$ is defined as $\tilde{f}_k(x)=g\circ f^k(x)$ for every $k\in\mathbb{N}_0$ and
\begin{align}
\label{vlnovkovy}
\tilde{X}_{n+1}=\tilde{f}_k(\tilde{X}_n) \text{ with probability } p_k,
\end{align}
for every $n\in\mathbb{N}_0$, $\tilde{X}_0$ is a given starting point or a random variable. The stationary distribution of the random process $\{\hat{Z}_n\}_{n=0}^\infty$ defined in the section $1.1$ can be written using the stationary distribution of $\{\tilde{X}_n\}_{n=0}^{\infty}$  under the assumptions  formulated in the following  claim.

\begin{claim}
\label{presnetvary}
Let $\{\tilde{X}_n\}_{n=0}^{\infty}$ be the stochastic process defined in (\ref{vlnovkovy}) with a unique stationary distribution $\tilde{\nu}$ and $\ppp\left(S\right)=1$. Then
\begin{itemize}
    \item the random proces $\{\hat{Z}_n\}_{n=1}^{\infty}$ has the stationary distribution $\hat{\mu}^\star$, which is of the form
    \begin{align}
    \label{hatmu}
\hat{\mu}^\star\left(\{0\}\times A\right)=\frac{\tilde{\nu}\left(A\right)}{1+E} \quad\text{ and}\quad\hat{\mu}^\star\left(\{k\}\times A\right)=\sum\limits_{j=0}^{\infty}\frac{p_{k+j}}{1+E}\tilde{\nu}\left(f^{-j-
1}\left(A\right)\right),
    \end{align}
for every $k\in\mathbb{N}$ and $A\in \B$,
\item the random process $\left\{X_n\right\}_{n=0}^{\infty}$ has the limit distribution $\nu$, which is of the form
    \begin{align*}
\nu\left( A\right)=\sum\limits_{k=0}^{\infty}\hat{\mu}^\star\left(\{k\}\times A\right),
    \end{align*}
for every $A\in \B$.    
\end{itemize}
\end{claim}

\begin{proof}\phantom{\qedhere}\hspace{-0.85cm}
The measure $\hat{\mu}^\star$ is a stationary distribution for the homogeneous Markov process $\{\hat{Z}_n\}_{n=1}^{\infty}$ iff
\begin{align}
\label{integral vztah}
\hat\mu^*(\hat{B})=\int\limits_{\mathbb{N}_0\times I}Pr(\hat{Z}_{2}\in \hat{B}\left|\hat{Z}_{1}=(k,x)\right.)\dd \hat\mu^*\left(k,x\right),
\end{align}
 for every $\hat{B}\in \hat{\mathcal{S}}$ (cf. \cite{BM} relation (9.1') on page 177).
 
First take $\hat{B}=\{0\}\times A$. Since $\hat{Z}_{2}\in \{0\}\times A$ if and only if $\hat{Z}_{1}\in \{0\}\times g^{-1}(A)$ or $\hat{Z}_{1}\in \{1\}\times g^{-1}(A)$, the state $\{0\}\times A$ can be reached either from the state $\{0\}\times g^{-1}(A)$ with probability $p_0$ or from the state $\{1\}\times g^{-1}(A)$ with probability $1$. Thus from (\ref{integral vztah}) we get
\begin{align}
\label{rv1}
\hat{\mu}^\star\left(\{0\}\times A\right)=p_0\hat{\mu}^\star\left(\{0\}\times g^{-1}(A)\right)+\hat{\mu}^\star\left(\{1\}\times g^{-1}(A)\right).
\end{align}

Now let $\hat{B}=\{k\}\times A$, where $k\in\mathbb{N}$. A set $\left\{k\right\}\times A$ for $k\in\mathbb{N}$ can be reached either from the set $\{0\}\times f^{-1}(A)$ with probability $p_k$ or from the set $\{k+1\}\times f^{-1}(A)$ with probability $1$.
Thus, from (\ref{integral vztah}) we obtain
\begin{align}
\label{rv2}
\hat{\mu}^\star\left(\{k\}\times A\right)=p_k\hat{\mu}^\star\left(\{0\}\times f^{-1}(A)\right)+\hat{\mu}^\star\left(\{k+1\}\times f^{-1}(A)\right).
\end{align}

Moreover, the stationary distribution $\tilde{\nu}$ of (\ref{vlnovkovy}) satisfies
\begin{align}
\label{s2}
\tilde{\nu}\left(A\right)=\sum\limits_{k=0}^{\infty}p_k\tilde{\nu}\left(\tilde{f}^{-1}_k(A)\right)=p_0\tilde{\nu}\left(g^{-1}(A)\right)+\sum\limits_{k=1}^{\infty}p_k\tilde{\nu}\left(f^{-k}\circ g^{-1}(A)\right).
\end{align}

Now it is possible to verify that the distribution $\hat{\mu}^\star$ given by (\ref{hatmu}) satisfies (\ref{rv1}) and (\ref{rv2}). The right-hand side of the equality (\ref{rv1}) can be written as
\begin{align*}
&p_0\hat{\mu}^\star\left(\{0\}\times g^{-1}(A)\right)+\hat{\mu}^\star\left(\{1\}\times g^{-1}(A)\right)\stackrel{(\ref{hatmu})}{=\joinrel=}\\
&\stackrel{(\ref{hatmu})}{=\joinrel=}p_0\frac{\tilde{\nu}\left(g^{-1}(A)\right)}{1+E}+\sum\limits_{j=0}^{\infty}\frac{p_{1+j}}{1+E}\tilde{\nu}\left(f^{-j-1}\circ g^{-1}\left(A\right)\right)=\\
&=\frac{1}{1+E}\left(p_0\tilde{\nu}\left(g^{-1}(A)\right)+\sum\limits_{k=1}^{\infty}p_{k}\tilde{\nu}\left(f^{-k}\circ g^{-1}(A)\right)\right)\stackrel{(\ref{s2})}{=\joinrel=}\\
&\stackrel{(\ref{s2})}{=\joinrel=}\frac{\tilde{\nu}(A)}{1+E}\stackrel{(\ref{hatmu})}{=\joinrel=}\hat{\mu}^\star\left(\{0\}\times A\right).
\end{align*} 
%The $\tilde{\pi}$ is stationary distribution of the (\ref{vlnovkovy}) so it must hold
%\begin{align}
%\label{s2}
%\tilde{\pi}\left(A\right)=\sum\limits_{k=0}^{\infty}p_k\tilde{\pi}\left(\tilde{f}^{-1}_k(A)\right)=p_0\tilde{\pi}\left(g^{-1}(A)\right)+\sum\limits_{k=1}^{\infty}p_k\tilde{\pi}\left(f^{-k}\circ g^{-1}(A)\right).
%\end{align}

%The state $\{k\}\times A$ for $k\in\mathbb{N}$ can be reached either from the state $\{0\}\times f^{-1}(A)$ with probability $p_k$ or from the state $\{k+1\}\times f^{-1}(A)$ with probability $1$.
%Thus, $\hat{\mu}^\star\left(\{k\}\times A\right)$ must hold
%\begin{align}
%\label{rv2}
%\hat{\mu}^\star\left(\{k\}\times A\right)=p_k\hat{\mu}^\star\left(\{0\}\times f^{-1}(A)\right)+\hat{\mu}^\star\left(\{k+1\}\times f^{-1}(A)\right).
%\end{align}

Similiarly, the right-hand side of the equality (\ref{rv2}) can be written as
\begin{align*}
&p_k\hat{\mu}^\star\left(\{0\}\times f^{-1}(A)\right)+\hat{\mu}^\star\left(\{k+1\}\times f^{-1}(A)\right)\stackrel{(\ref{hatmu})}{=\joinrel=}\\
&\stackrel{(\ref{hatmu})}{=\joinrel=}p_k\frac{\tilde{\nu}\left(f^{-1}(A)\right)}{1+E}+\sum\limits_{j=0}^{\infty}\frac{p_{k+1+j}}{1+E}\tilde{\nu}\left(f^{-j-1}\circ f^{-1}\left(A\right)\right)\stackrel{}{=}\\
&\stackrel{}{=}\frac{1}{1+E}\left(p_k\tilde{\nu}\left(f^{-1}(A)\right)+\sum\limits_{l=1}^{\infty}p_{k+l}\tilde{\nu}\left(f^{-l}\circ f^{-1}\left(A\right)\right)\right)\stackrel{}{=}\\
&\stackrel{}{=}\frac{1}{1+E}\left(\sum\limits_{l=0}^{\infty}p_{k+l}\tilde{\nu}\left(f^{-l-1}\left(A\right)\right)\right)\stackrel{(\ref{hatmu})}{=\joinrel=}\hat{\mu}^\star\left(\{k\}\times A\right).
\end{align*} 
Hence we proved that $\hat{\mu}^\star$ given by (\ref{hatmu})  is a stationary distribution for the stochastic process $\{\hat{Z}_n\}_{n=1}^\infty$. 

Since  $\ppp(S)=1$, the two-dimensional stochastic process $\{\hat{Z}_n\}_{n=1}^{\infty}$ where $\hat{Z}_n\left(\omega\right)=\left(\omega_{n-1}, X_n\left(\omega\right)\right)$ satisfies Theorem \ref{thm1}. It means that for any initial distribution $\hat\mu$ the sequence of measures $T^n\hat{\mu}$ converges weakly to $\hat{\pi}\odot \ppp$. Moreover, the Markov process $\{\hat{Z}_n\}_{n=1}^{\infty}$ has the stationary distribution $\hat{\mu}^\star$ so $\hat{\pi}\odot \ppp=\hat{\mu}^\star$. Applying Continuous Mapping Theorem and considering the continuity of the projection onto the second coordinate, we conclude that the stochastic process $\left\{X_n\right\}_{n=0}^{\infty}$ also has a limit distribution~$\nu$, which takes the form
\begin{align*}
\nu\left(A\right)=\sum\limits_{k=0}^{\infty}\hat{\mu}^\star\left(\left\{k\right\}\times A \right).\tag*{\qed}
\end{align*}
\end{proof}

In the following example, we will show the use of Claim \ref{presnetvary}. We will assume that the time between impulses $T_i$ has the Bernoulli distribution with parameter $p=\frac{1}{2}$ for every $i\in\N$, despite the fact that in the whole paper we assume that $T_i$ takes  infinitely but countably many values. The Bernoulli distribution was chosen because in this case the stationary distribution of the stochastic process $\{\tilde{X}_n\}_{n=0}^{\infty}$ is known.
%in (\ref{vlnovkovy})
\begin{example}
Let the functions $f,g:\left<0,2\right>\to \left<0,2\right> $ be defined as $g(x)=\frac{x}{2}$ and $f(x)=x+1$ for $x\in\left<0,1\right>$ and $f(x)=2$ for $x\in \left(1,2\right>$. Assume that the time between impulses $T_i$ has the Bernoulli distribution with paramater $p=\frac{1}{2}$ for every $i\in\N$. The function $g$ maps the interval $\left<0,2\right>$ to the interval $\left<0,1\right>$, therefore the system (\ref{vlnovkovy}) will operate only on the interval $\left<0,1\right>$ and from the second iteration onwards will have the form
\begin{align}
\label{vlnovkovy1}
\tilde{X}_{n+1}=\begin{cases}
g(\tilde{X}_{n})=\frac{\tilde{X}_{n}}{2} &\text{with probability }\frac{1}{2},\\
g(f(\tilde{X}_{n}))=\frac{\tilde{X}_{n}+1}{2} &\text{with probability }\frac{1}{2},
\end{cases}
\end{align}
regardless of the form of the function $f$ on the interval $\left(1,2\right>$.

 It is well known that the uniform distribution on the interval $\left<0,1\right>$, denote it as $\tilde{\nu}$, is a unique stationary distribution of the stochastic process define in (\ref{vlnovkovy1}) (cf.~\cite{BM} Exercise 4.1 on page 248). 

From (\ref{hatmu}), the stationary distribution $\hat{\mu}^\star$ of the stochastic proces $\{\hat{Z}_n\}_{n=1}^\infty$ has the form
\begin{align*}
\hat{\mu}^\star\left(\{0\}\times A\right)=\frac{\tilde{\nu}\left(A\right)}{1+E} \quad \text{and} \quad \hat{\mu}^\star\left(\{1\}\times A\right)=\frac{p_{1}}{1+E}\tilde{\nu}\left(f^{-1}\left(A\right)\right)
\end{align*}
for $A\in \mathcal{B}(\left<0,2\right>)$, where $p_1=p$ and $E\equiv E(T_1)=\frac{1}{2}$. The Lipschitz constant $L_1$ of the function $f$ is $1$. The Lipschitz constant $L_0$ of the function $g$ is $\frac{1}{2}$.  It is possible to use the relation~(\ref{claim1})  because $1<\sqrt{2}$. This implies that by Theorem~\ref{thm1} the sequence $T^n\hat{\mu}$ convergences weakly to $\hat{\pi}\odot \ppp$. Moreover $\{\hat{Z}_n\}_{n=1}^\infty$ has stationary distribution $\hat{\mu}^*$ so $\hat{\mu}^*=\hat{\pi}\odot \ppp$. From Claim \ref{presnetvary} we obtain that the limit distribution of the process $\left\{X_n\right\}_{n=0}^{\infty}$ is
\begin{align*}
\nu(A)=\hat{\mu}^\star\left(\{0\}\times A\right)+\hat{\mu}^\star\left(\{1\}\times A\right)=\frac{\tilde{\nu}\left(A\right)}{1+E} +\frac{p_{1}}{1+E}\tilde{\nu}\left(f^{-1}\left(A\right)\right).
\end{align*}

Now it is easy to obtain the distribution function $F$ of the limit distribution as $F(a)= \nu (\left(-\infty, a\right>)$. For the uniform distribution $\tilde\nu$ we have 
\begin{align*}
\tilde\nu(\left(-\infty, a\right>)=
\begin{cases}
0 & \text{if } a \le 0,\\
a & \text{if } a \in (0,1),\\
1 & \text{if } a \ge 1,
\end{cases}
\end{align*}
therefore 
\begin{itemize}
    \item if  $a\in\left(-\infty,0\right)$ then
\begin{align*}
\textstyle F(a)=\tilde{\nu}\left(\left(-\infty, a\right>\right)\cdot\frac{2}{3} +\tilde{\nu}\left(f^{-1}(\left(-\infty, a\right>)\right)\cdot\frac{1}{3}=0 + \tilde\nu(\emptyset)\cdot\frac{1}{3}=0,
\end{align*}
    \item if  $a\in\left<0,1\right)$ then
\begin{align*}
\textstyle F(a)&=\textstyle 
\tilde{\nu}\left(\left(-\infty,a\right>\right)\cdot\frac{2}{3} +\tilde{\nu}\left(f^{-1}(\left(-\infty,a\right>)\right)\cdot\frac{1}{3}=\textstyle \frac{2}{3}a+\tilde{\nu}\left(\emptyset\right)\cdot\frac{1}{3}=\frac{2}{3}a,
\end{align*}
\item if $a\in\left<1,2\right)$ then
\begin{align*}
\textstyle F(a)&=\textstyle \tilde{\nu}\left(\left(-\infty,a\right>\right)\cdot\frac{2}{3} +\tilde{\nu}\left(f^{-1}(\left(-\infty,a\right>)\right)\cdot\frac{1}{3}=1\cdot\frac{2}{3}+\tilde{\nu}\left(\left<0,a-1\right>\right)\cdot\frac{1}{3}=\\
&=\textstyle \frac{2}{3}+\frac{1}{3}(a-1)=\frac{1}{3}a+\frac{1}{3},
\end{align*}
\item if $a\in\left<2,\infty\right)$ then
\begin{align*}
\textstyle 
\textstyle F(a)&=\textstyle\tilde{\nu}\left(\left(-\infty,a\right>\right)\cdot\frac{2}{3} +\tilde{\nu}\left(f^{-1}(\left(-\infty,a\right>\right)\cdot\frac{1}{3}=1\cdot\frac{2}{3}+\tilde{\nu}\left(\left<0,2\right>\right)\cdot\frac{1}{3}=\\
&=\textstyle \frac{2}{3}+\frac{1}{3}=1.
\end{align*}
\end{itemize}
%(cf. \cite{BM} Exercise 4.1).

In Figure \ref{obrazok9}, we demonstrate the convergence of the random process (\ref{vlnovkovy1}) to its limit distribution. We compare the empirical distribution function with the distribution function corresponding to the limit distribution $\nu$. The empirical distribution function is depicted for only $1000$ realizations in Figure \ref{obrazok9}a, while in Figure \ref{obrazok9}b it is depicted for $80000$ realizations.

\begin{figure}[h!]
\centering
\subfloat[Empirical distribution function for $1000$\\ realizations.]{%
\resizebox*{6.1cm}{!}{\includegraphics{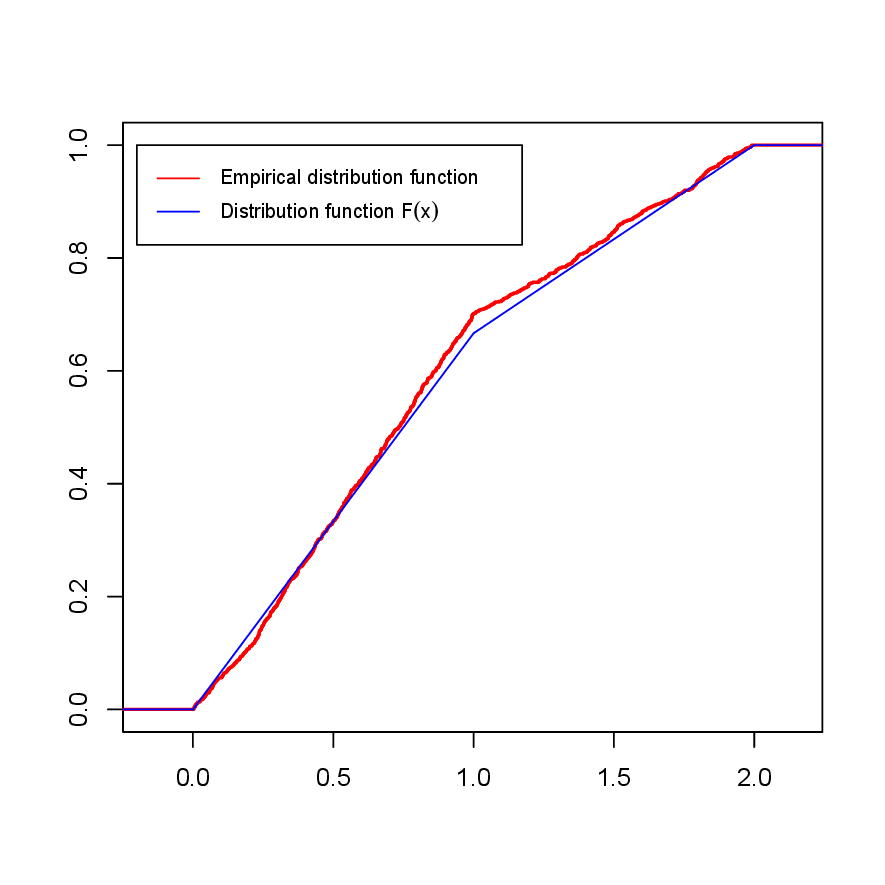}}}
\hspace{5pt}
\subfloat[Empirical distribution function for $80000$ realizations.]{%
\resizebox*{6.1cm}{!}{\includegraphics{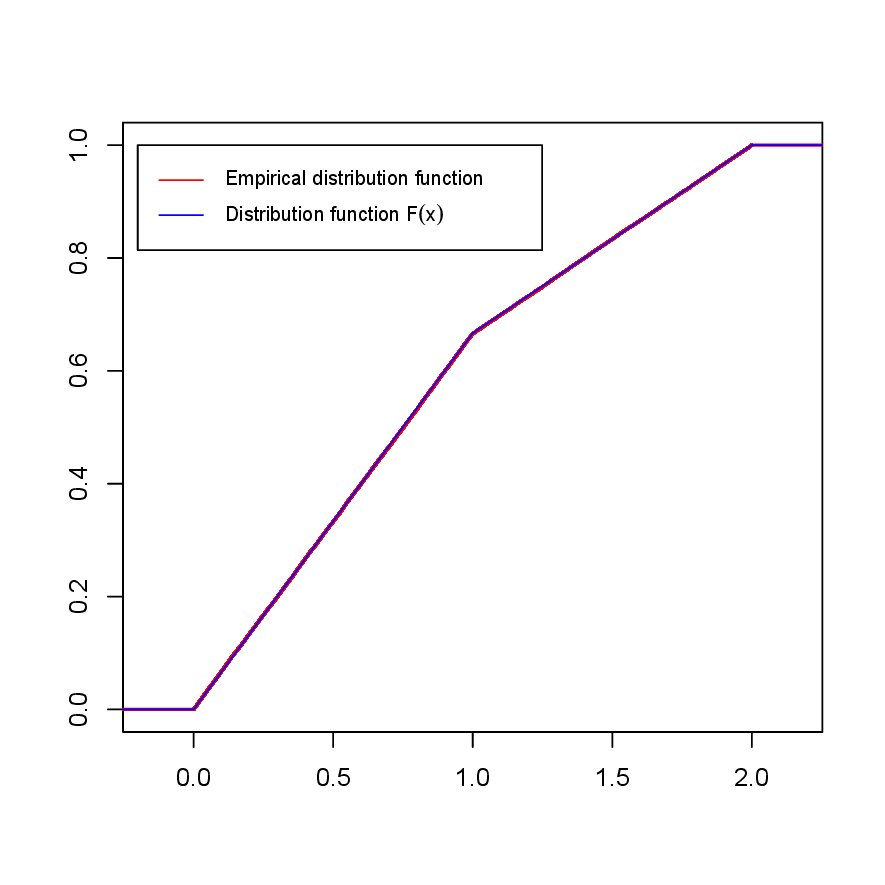}}}
\caption{Comparison of the empirical distribution functions with the distribution function $F$ of of the limit distribution $\nu$.} 
\label{obrazok9}
\end{figure}
\end{example}

\section{Appendix}
\label{priloha}
\begin{proof}[Proof of Lemma \ref{L1}]
The set $\hat{M}=\mathbb{N}_0\times I$ can be written in the form $\hat{M}=\bigcup\limits_{j=0}^{\infty}\left\{j\right\}\times~I$. If we use the $\sigma\text{-additivity}$ of the Lebesgue integral, we get
\begin{align*}
\int\limits_{\hat{M}} \hat{h}\dd\hat\nu = \sum_{j=0}^\infty \int\limits_{\{j\}\times I} \hat{h} \dd\hat{\nu}.
\end{align*}
To prove the lemma it suffices to show
\begin{align}
\label{mmiery}
\int\limits_{\{j\}\times I}\hat{h}\dd\hat{\nu}=\int\limits_I h_j \dd\nu_j,
\end{align}
where $h_j(x)\equiv \hat{h}(j,x)$ and $\nu_j(C)\equiv \hat{\nu}(\{j\}\times C)$ for any $C\in \mathcal{B}(I)$. Assume that $\hat{h}=\chi_{\hat{B}}$ is the indicator function of a set $\hat{B}\in \mathcal{\hat{S}}$. Then $\hat{B}=\bigcup\limits_{j=0}^\infty \{j\}\times B_j$ where $B_j=\{x\in I: (j,x)\in \hat{B}\}$ and
\begin{align*}
\int\limits_{\{j\}\times I }\chi_{\hat{B}}\dd\hat{\nu}=\int\limits_{\hat{M}} \chi_{\hat{B}\cap(\{j\}\times I)}\dd\hat{\nu}=\hat{\nu}(\{j\}\times B_j)=\nu_j(B_j)=\int\limits_I \chi_{B_j} \dd\nu_j ,
\end{align*}
hence (\ref{mmiery}) holds. Further, we can proceed using standard techniques of the measure theory and show that (\ref{mmiery}) holds for any continuous bounded  function.
%We can use change of variables for pushforward measure at the integral $\int\limits_{\{j\}\times I} \hat{g}(k,x) \dd \hat{\nu}(k,x)$. Let's define for fixed $j\in\N_0$ a mapping $\tau_j:\{j\}\times I \to I$ as $\tau_j\left(k,x\right)=x$. Then $\tau_j^{-1}\left(C\right)=\{j\}\times C$ for every $C\in\B$. It means that $\tau_j^{-1}\left(I\right)=\{j\}\times I$. Moreover the $\hat{g}(j,x)=g_j\circ\tau_j(k,x)$ where $g_j(x)\equiv\hat{g}(j,x)$ and $\hat\nu(\tau_j^{-1}(C))=\hat\nu(\{j\}\times C)$. The function $\hat{g}(k,x)$ define on the set $\{j\}\times I$ takes the same values as a function $g_j(x)\equiv\hat{g}(j,x)$ defined on the set $I$. Similarly, the measure $\hat\nu$ defined on $(\{j\}\times I,\{j\}\otimes\B)$ takes the same value as the measure $\nu_j(C)\equiv\hat\nu(\{j\}\times C)$ defined on $(I,\B)$. Hence we have 
%\begin{align*}
%\int\limits_{\{j\}\times I} \hat{g}(k,x) \dd \hat{\nu}(k,x)=\int\limits_{I} g_j(x) \dd \nu_j(x).
%\end{align*}
%This concludes the proof.
\end{proof}

%\begin{proof}[Proof of Lemma \ref{L3}]
%
%\end{proof}

\begin{proof}[Proof of Lemma \ref{L4}]
Let $T$ be the transition operator of the Markov process $\{\hat{Z}_n\}_{n=1}^{\infty}$, which is given by a transition matrix $P=\left(p_{ik}\right)_{i,k=0}^{\infty}$ and by  functions $\{f_i\}_{i=0}^{\infty}$. Then for any measure $\hat{\mu}$ defined on $\hat{\mathcal{S}}$ and any $\hat{B}\in\hat{\mathcal{S}}$ we have
\begin{align}
\label{operator T1}
T\hat{\mu}(\hat{B})=\sum\limits_{i=0}^{\infty}\sum\limits_{k=0}^{\infty}p_{ik}\mu_i\left(f_k^{-1}\left(B_k\right)\right).
\end{align}
Note that in our case for a fixed $k\in\mathbb{N}_0$ only $p_{0k}=p_k$ and $p_{(k+1)k}=1$ are non-zero and thus the relation (\ref{operator T1}) is consistent with (\ref{operator T}). However, in this proof it will be easier to use the general form (\ref{operator T1}). 

The proof will proceed by induction. Let's denote $\hat{B}\equiv \left\{j\right\}\times C$, where $j\in\mathbb{N}_0$ and $C\in\mathcal{B}\left(I\right)$. For $n=1$ we obtain
\begin{align}
\label{indukcia1}
\begin{split}
\left(T^n\hat{\mu}\right)_j\left(C\right)&=\left(T\hat{\mu}\right)_j\left(C\right)=T\hat{\mu}(\hat{B})=\sum\limits_{i=0}^{\infty}\sum\limits_{k=0}^{\infty}p_{ik}\mu_i\left(f_k^{-1}\left(B_k\right)\right)\stackrel{(a)}{=}\\
&\stackrel{(a)}{=}\sum\limits_{i=0}^{\infty}p_{ij}\mu_i\left(f_j^{-1}\left(C\right)\right)\stackrel{(b)}{=}\sum\limits_{\xi_1=0}^{\infty}p_{\xi_1j}\cdot f_j\odot\mu_{\xi_1}\left(C\right),
\end{split}
\end{align}
where in the equality $(a)$ we used that $B_k=\emptyset$ if $k\neq j$ and $B_k=C$ if $k=j$ and in the equality $(b)$ we used that the operator $\odot$ is defined as $f_j\odot\mu_i\left(C\right)=\mu_i(f_j^{-1}\left(C\right))$ and we also changed the summation index $i$~to~$\xi_1$. So for $n=1$ we obtain (\ref{vzL4}) from Lemma \ref{L4}. 

Using the induction assumption we obtain that
\begin{align}
\label{hviezdicka}
\begin{split}
&\left(T^{n+1}\hat{\mu}\right)_j\left(C\right)=T^{n+1}\hat{\mu}(\hat{B})=T^{n}T\hat{\mu}(\hat{B})=\\
&=\sum\limits_{\xi_{n}=0}^{\infty}\ldots\sum\limits_{\xi_1=0}^{\infty}p_{\xi_{n}\xi_{n-1}}\cdot\ldots\cdot p_{\xi_2\xi_1}\cdot p_{\xi_1 j}\cdot\left(f_j\circ f_{\xi_1}\circ\ldots\circ f_{\xi_{n-1}}\right)\odot\left(T\hat{\mu}\right)_{\xi_n}\left(C\right).
\end{split}
\end{align}

The expression $\left(T\hat{\mu}\right)_{\xi_n}\left(C\right)$ can be written down in the same way as we did it in~(\ref{indukcia1}). It means that 
\begin{align}
\label{vztah mu}
\left(T\hat{\mu}\right)_{\xi_n}\left(C\right)=\sum\limits_{\xi_{n+1}=0}^{\infty}p_{\xi_{n+1}\xi_{n}}\cdot f_{\xi_n}\odot \mu_{\xi_{n+1}}\left(C\right).
\end{align}

Since $h_1\odot h_2\odot \mu=\left(h_1\circ h_2\right)\odot \mu$ for any measurable functions $h_1$, $h_2$ and measure $\mu$, we can put (\ref{vztah mu}) in  (\ref{hviezdicka}) to obtain
\footnotesize
\begin{align*}
&\left(T^{n+1}\hat{\mu}\right)_j\left(C\right)=\ldots=\\
&=\sum\limits_{\xi_{n+1}=0}^{\infty}\sum\limits_{\xi_{n}=0}^{\infty}\ldots\sum\limits_{\xi_1=0}^{\infty}p_{\xi_{n+1}\xi_{n}}p_{\xi_{n}\xi_{n-1}}\cdot\ldots\cdot p_{\xi_2\xi_1}\cdot p_{\xi_1 j}\cdot\left(f_j\circ f_{\xi_1}\circ\ldots\circ f_{\xi_{n-1}}\circ f_{\xi_{n}}\right)\odot \mu_{\xi_{n+1}}\left(C\right).
\end{align*}
\normalsize
Since there are a lot of zeros in the transition matrix $P$, there are only finitely many non-zero terms in the sums for fixed $n\in\mathbb{N}$ and fixed $j\in\mathbb{N}_0$. This enables the change of the summation order.
%Changing the summation order is a correct step in this case because there are only finitely many non-zero terms in the sums for fixed $n\in\mathbb{N}$ and fixed $j\in\mathbb{N}_0$, since there are a lot of zeros in our transition matrix $P$. More precisely, $p_{\xi_{n}\xi_{n-1}}\cdot\ldots\cdot p_{\xi_{2}\xi_{1}}\cdot p_{\xi_{1}j}$ is probability of the trajectory $\left(\xi_n,\xi_{n-1},\ldots,\xi_{1},j\right)$ and such trajectories which are admissible in our chain with the transition matrix $P$, are only finitely many. So this concludes the proof.
 More precisely, $p_{\xi_{n}\xi_{n-1}}\cdot\ldots\cdot p_{\xi_{2}\xi_{1}}\cdot p_{\xi_{1}j}$ is the probability of the trajectory $\left(\xi_n,\xi_{n-1},\ldots,\xi_{1},j\right)$ and the number of  trajectories of the length $n+1$ ending in the state $j$ which are admissible in the chain with the transition matrix $P$ is finite. This concludes the proof.
\end{proof}

\begin{proof}[Proof of Corollary \ref{C1}]\phantom{\qedhere}\hspace{-0.85cm}
From Lemma \ref{L4} we have
\begin{align*}
\left(T^n\hat{\mu}\right)_j\left(I\right)=\sum\limits_{\xi_{n}=0}^{\infty}\ldots\sum\limits_{\xi_{1}=0}^{\infty}p_{\xi_{n}\xi_{n-1}}\cdot\ldots\cdot p_{\xi_{2}\xi_{1}}\cdot p_{\xi_{1}j} \cdot \left(f_j\circ f_{\xi_{1}}\circ \ldots\circ f_{\xi_{n-1}}\right)\odot \mu_{\xi_n}\left(I\right).
\end{align*}
Moreover, $f_{i}^{-1}\left(I\right)=I$ for every $i\in\mathbb{N}_0$, so 
\begin{align*}
\left(f_j\circ f_{\xi_{1}}\circ \ldots\circ f_{\xi_{n-1}}\right)\odot \mu_{\xi_n}\left(I\right)=\mu_{\xi_n}\left(\left(f_j\circ f_{\xi_{1}}\circ \ldots\circ f_{\xi_{n-1}}\right)^{-1}\left(I\right)\right)=\mu_{\xi_n}\left(I\right).
\end{align*}
Therefore, we can write
\begin{align*}
\left(T^n\hat{\mu}\right)_j\left(I\right)=\sum\limits_{\xi_{n}=0}^{\infty}\ldots\sum\limits_{\xi_{1}=0}^{\infty}p_{\xi_{n}\xi_{n-1}}\cdot\ldots\cdot p_{\xi_{2}\xi_{1}}\cdot p_{\xi_{1}j} \cdot \mu_{\xi_n}\left(I\right)=\sum\limits_{i=0}^{\infty}\mu_i\left(I\right)\cdot p_{ij}^{(n)},
\end{align*}
where $p_{ij}^{(n)}$ denotes the $n\text{-step}$ transition probability from the state $i$ to the state $j$ and for simplicity, the index was changed from $\xi_{n}$ to $i$. 

It means that we want to prove that 
\begin{align}
\label{limit}
\lim\limits_{n\to\infty}\sum\limits_{j=0}^{\infty}\left|\sum\limits_{i=0}^{\infty}\mu_i\left(I\right)p_{ij}^{(n)}-m_j\right|=0,
\end{align}

The limit in (\ref{limit}) is definitely greater or equal to zero because each its term is nonnegative. Moreover $\mu_i\left(I\right)=\hat{\mu}\left(\left\{i\right\}\times I\right)$ and $\hat{\mu}$ is a probability measure, so $\sum\limits_{i=0}^{\infty}\mu_i\left(I\right)=~1$. If we use this fact we have
\begin{align*}
&\lim\limits_{n\to\infty}\sum\limits_{j=0}^{\infty}\left|\sum\limits_{i=0}^{\infty}\mu_i\left(I\right)p_{ij}^{(n)}-m_j\right|=\lim\limits_{n\to\infty}\sum\limits_{j=0}^{\infty}\left|\sum\limits_{i=0}^{\infty}\mu_i\left(I\right)p_{ij}^{(n)}-\sum\limits_{i=0}^{\infty}\mu_i\left(I\right)m_j\right|=\\
&=\lim\limits_{n\to\infty}\sum\limits_{j=0}^{\infty}\left|\sum\limits_{i=0}^{\infty}\mu_i\left(I\right)\left(p_{ij}^{(n)}-m_j\right)\right|\leq \lim\limits_{n\to\infty}\sum\limits_{j=0}^{\infty}\sum\limits_{i=0}^{\infty}\mu_i\left(I\right)\left|p_{ij}^{(n)}-m_j\right|\stackrel{(a)}{=}\\
&\stackrel{(a)}{=}\lim\limits_{n\to\infty}\sum\limits_{i=0}^{\infty}\sum\limits_{j=0}^{\infty}\mu_i\left(I\right)\left|p_{ij}^{(n)}-m_j\right|\stackrel{(b)}{=}\sum\limits_{i=0}^{\infty}\mu_i\left(I\right)\lim\limits_{n\to\infty}\sum\limits_{j=0}^{\infty}\left|p_{ij}^{(n)}-m_j\right|\stackrel{(c)}{=}0.
\end{align*}
Individual steps are more precisely justified in the following:
\begin{itemize}
	\item To change the order of the summations in the equality $(a)$, we have to show that 
	\begin{align*}
  \sum\limits_{j=0}^{\infty}\sum\limits_{i=0}^{\infty}\mu_i\left(I\right)\left|p_{ij}^{(n)}-m_j\right|<\infty \quad  \text{ or } \quad   \sum\limits_{i=0}^{\infty}\sum\limits_{j=0}^{\infty}\mu_i\left(I\right)\left|p_{ij}^{(n)}-m_j\right|<\infty.
	\end{align*}
We are going to show that $\sum\limits_{i=0}^{\infty}\sum\limits_{j=0}^{\infty}\mu_i\left(I\right)\left|p_{ij}^{(n)}-m_j\right|<\infty$, so
\begin{align*}
 &\sum\limits_{i=0}^{\infty}\sum\limits_{j=0}^{\infty}\mu_i\left(I\right)\left|p_{ij}^{(n)}-m_j\right|=\sum\limits_{i=0}^{\infty}\sum\limits_{j=0}^{\infty}\left|\mu_i\left(I\right)p_{ij}^{(n)}-\mu_i\left(I\right)m_j\right|\leq\\
&\leq\sum\limits_{i=0}^{\infty}\sum\limits_{j=0}^{\infty}\left|\mu_i\left(I\right)p_{ij}^{(n)}\right|+\sum\limits_{i=0}^{\infty}\sum\limits_{j=0}^{\infty}\left|\mu_i\left(I\right)m_j\right|=\sum\limits_{i=0}^{\infty}\mu_i\left(I\right)\sum\limits_{j=0}^{\infty}\left|p_{ij}^{(n)}\right|+1=\\
&=1+1=2<\infty.
\end{align*}
		\item To change the limit and the sum in the equality $(b)$, we have used Lebesgue's theorem. Now we check that its assumptions are satisfied. If we take the sequence
		\small
			\begin{align*}
 \left\{\sum\limits_{j=0}^{\infty}\mu_i\left(I\right)\left|p_{ij}^{(n)}-m_j\right|\right\}_{n=0}^{\infty},
			\end{align*}		
			\normalsize
		we have $\lim\limits_{n\to\infty}\sum\limits_{j=0}^{\infty}\mu_i\left(I\right)\left|p_{ij}^{(n)}-m_j\right|=0$ for every $i\in\N_0$ (cf. \cite{BM} Theorem 8.2 on page 172). We also have that 
	\begin{align*}		
		&\left|\sum\limits_{j=0}^{\infty}\mu_i\left(I\right)\left|p_{ij}^{(n)}-m_j\right|\right|=\sum\limits_{j=0}^{\infty}\mu_i\left(I\right)\left|p_{ij}^{(n)}-m_j\right|\leq\\
		&\leq\mu_i\left(I\right)\left(\sum\limits_{j=0}^{\infty}\left|p_{ij}^{(n)}\right|+\sum\limits_{j=0}^{\infty}\left|m_j\right|\right)=2\mu_i\left(I\right)
	\end{align*}
		for every $n\in\N$. Since  $\sum\limits_{i=0}^{\infty}2\mu_i\left(I\right)=2$ the assumptions of Lebesgue's theorem are satisfied.
		\item In the equality $(c)$ we used Theorem 8.2 from \cite{BM} on page~172. \qed
\end{itemize}
\end{proof}

%\begin{lem}(\cite{DiazMatias} Lemma 5.2.)
%\label{L52}
%For any sequence $\{\nu_n\}_{n=0}^\infty$ of probability measures on $\left(I,\mathcal{B}(I)\right)$, for any $\omega\in S$ and for any continuous function $g:I\to\R$
%\begin{align*}
%\lim_{n\to\infty} \int\limits_{I} g \dd  \left(f_{\omega_0}\circ\ldots\circ f_{\omega_n}\right)\odot \nu_n=g\circ \pi(\omega).
%\end{align*}
%\end{lem}

\begin{proof}[Proof of Lemma \ref{L52}]
It is sufficient to demonstrate that for any $\varepsilon>0$, there exists $n_0\in\mathbb{N}$ such that 
\begin{align*}
\left|h\circ \pi(\omega)-\int\limits_{I} h \dd\left(f_{\omega_0}\circ\ldots\circ f_{\omega_n}\right)\odot \nu_n\right|\le\varepsilon
\end{align*}
for every $n\ge n_0$. Let $\varepsilon>0$. Since the function $h$ is continuous, there exists $\delta>0$ such that if  $|y-\pi(\omega)|<\delta$,  then $|h(y)-h\circ\pi(\omega)|<\varepsilon$. Further $\pi(\omega)$ is defined as $\lim\limits_{n\to\infty} f_{\omega_0}\circ\ldots\circ f_{\omega_n}(x)$ for any $x$, since this limit does not depend on it.

Thus, for $\delta>0$, there exists some $n_0$, such that $\left|f_{\omega_0}\circ \ldots\circ f_{\omega_n}(x)-\pi(\omega)\right|<\delta$ for $n\ge n_0$ and for every $x\in I$. For $n\ge n_0$, we obtain that
\begin{equation}
\label{epsilon}
|h\circ \pi(\omega)-g\circ f_{\omega_0}\circ\ldots\circ f_{\omega_n}(x)|<\varepsilon
\end{equation} 
for every $x\in I$. Finally we have
\small
\begin{align*}
&\left|h\circ \pi(\omega)-\int\limits_{I} h \dd  \left(f_{\omega_0}\circ\ldots\circ f_{\omega_n}\right)\odot\nu_n\right|=
\left|\int\limits_{I} h\circ \pi(\omega)\dd \nu_n-\int\limits_{I} h \dd \left(f_{\omega_0}\circ\ldots\circ f_{\omega_n}\right)\odot\nu_n\right|=\\
&=\left|\int\limits_{I} h\circ \pi(\omega)\dd \nu_n-\int\limits_{I} h\circ f_{\omega_0}\circ\ldots\circ f_{\omega_n} \dd \nu_n\right|\le  \int\limits_{I} |h\circ \pi(\omega)-h\circ  f_{\omega_0}\circ\ldots\circ f_{\omega_n}|\dd \nu_n\stackrel{(a)}{\leq} \\
&\stackrel{(a)}{\leq} \int\limits_{I} \varepsilon \ \mathrm{d}\nu_n= \varepsilon,
\end{align*}
\normalsize
where in the inequality $(a)$ we used (\ref{epsilon}).
\end{proof}

\begin{lemma}(\cite{DiazMatias} Proposition 3.1)
\label{DMprop31} 
Consider two $\ppp\text{-admissible}$ cylinders $W\equiv[\xi_0,\ldots,\xi_{N-1}]$ and $[\eta_0,\ldots\eta_{N-1}]$, for which
\begin{itemize}
\item [(i)] $\xi_0=\eta_0$,
\item [(ii)]  $\xi_{N-1}=\eta_{N-1}$,
\item [(iii)]  $\ppp([\xi_0,\ldots,\xi_{N-1}])\le \ppp([\eta_0,\ldots,\eta_{N-1}])$,
\item [(iv)]  $f_{\xi_0}\circ\ldots\circ f_{\xi_{N-1}}(I)\cap f_{\eta_0}\circ\ldots\circ f_{\eta_{N-1}}(I)=\emptyset$.
\end{itemize}
Then for any $n\in\mathbb{N}$ and for any $x\in I$
\begin{align*}
\ppp(S_{nN}^x)\le\ppp(\Sigma_n^W).
\end{align*}
\end{lemma} 
The sets $S_{nN}^x$ and $\Sigma_n^W$ contain sequences that form unions of the following classes of cylinders, which will be defined for a fixed $x\in I$ and fixed $n\in\mathbb{N}$ as
\begin{align*}
\Sigma_x^n&\equiv \{[a_0,\ldots,a_{nN-1}]\subseteq \Sigma: x\in f_{a_0}\circ\ldots\circ f_{a_{nN-1}}(I)\},\\
E^n&\equiv\{[a_0,\ldots,a_{nN-1}]\subseteq \Sigma: \sigma^{iN}([a_0,\ldots,a_{nN-1}])\cap W=\emptyset, i=0,1,\ldots,n-1\}.
\end{align*}
More precisely, we obtain
\begin{align}
S_{nN}^x&=\bigcup_{C\in \Sigma_x^n} C, \label{SnNxDecomp1}\\
\Sigma_n^W&=\bigcup_{C\in E^n} C,  \label{SigmanWDecomp1}
\end{align}
which implies 
\begin{align}
\ppp(S_{nN}^x)&=\sum_{C\in \Sigma_x^n} \ppp(C), \label{SnNxDecomp}\\
\ppp(\Sigma_n^W)&= \sum_{C\in E^n} \ppp(C), \label{SigmanWDecomp}
\end{align}
because $\Sigma_x^n$ and $E^n$ are countable sets. 

Next, for every $n\in\mathbb{N}$, we define the function $H_n:\Sigma_x^n\to E^n$ such that for $C=C_0\star\ldots\star C_{n-1}$
\begin{align*}
H_n(C)=H_n(C_0\star\ldots\star C_{n-1})=C_0^\prime\star\ldots\star C_{n-1}^\prime, 
\end{align*}
where $C_i^\prime=C_i$ if $C_i\neq[\xi_0,\ldots,\xi_{N-1}]$ and $C_i^\prime=[\eta_0,\ldots,\eta_{N-1}]$ if $C_i=[\xi_0,\ldots,\xi_{N-1}]$.

\begin{claim}(\cite{DiazMatias} Claim 3.2.)
\label{claim0}
For every $C\in \Sigma_x^n$ it holds $\ppp(C)\le \ppp(H_n(C))$ .
\end{claim}

\begin{proof}
We know that $\ppp([\xi_0,\ldots,\xi_{N-1}])=m_{\xi_0}\cdot q_{\xi_0\xi_1}\cdot\ldots\cdot q_{\xi_{N-2}{\xi_{N-1}}}$ and $\ppp([\eta_0,\ldots,\eta_{N-1}])=m_{\eta_0}\cdot q_{\eta_0\eta_1}\cdot\ldots\cdot q_{\eta_{N-2}{\eta_{N-1}}}$. According to the assumptions of Lemma~\ref{DMprop31}, $\ppp([\xi_0,\ldots,\xi_{N-1}])\le \ppp([\eta_0,\ldots,\eta_{N-1}])$ and $\xi_0=\eta_0$, thus also
\begin{align*}
q_{\xi_0\xi_1}\cdot\ldots\cdot q_{\xi_{N-2}\xi_{N-1}}\le q_{\eta_0\eta_1}\cdot\ldots\cdot q_{\eta_{N-2}\eta_{N-1}}.
\end{align*}
We can notice that if we apply the function $H_n$ to some cylinder $C=[a_0,\ldots,a_{nN-1}]$ (which has a probability of $m_{a_0}\cdot q_{a_0a_1}\cdot\ldots\cdot q_{a_{nN-2}a_{nN-1}}$), some sequence $\xi_0,\ldots,\xi_{N-1}$ may change to the sequence $\eta_0,\ldots,\eta_{N-1}$, otherwise, everything remains the same. The difference in the calculation of the probabilities of $H_n(C)$ and $C$ is that instead of certain products of the numbers $q_{\xi_0\xi_1}\cdot\ldots\cdot q_{\xi_{N-2}\xi_{N-1}}$, we will have the product of the numbers $q_{\eta_0\eta_1}\cdot\ldots\cdot q_{\eta_{N-2}\eta_{N-1}}$, which will not be less. Since we assume that  $\xi_0=\eta_0$ a $\xi_{N-1}=\eta_{N-1}$, it would not happen that $C$ is $\ppp\text{-admissible}$ while $H_n(C)$ is not.
\end{proof}

\begin{claim}(\cite{DiazMatias} Lemma 3.3.)
\label{claim2}
The function $H_n$ is injective for every $n\in\mathbb{N}$. 
\end{claim}

\begin{proof}
Assume the contrary - let there exist cylinders $C=C_0\star \ldots\star C_{n-1}$ and $\tilde C= \tilde{C}_0\star\ldots\star\tilde{C}_{n-1}$ such that  $C,\tilde C\in\Sigma_x^n$, $C\neq \tilde C$ and $H_n(C)=H_n(\tilde{C})$. Since $C\neq\tilde C$, there must exist some $i=0,1,\ldots,n-1$, such that $C_i\neq \tilde C_i$. Take the smallest $i$ with this property. We can assume that $C_i=[\xi_0,\ldots,\xi_{N-1}]$ and $\tilde C_i=[\eta_0,\ldots,\eta_{N-1}]$. Since we have taken the smallest possible $i$, the cylinders $C$ and $\tilde C$ must have the form $C=[a_0,\ldots,a_l,\xi_0,\ldots,\xi_{N-1},b_0,\ldots,b_r]$ and $\tilde C=[a_0,\ldots,a_l,\eta_0,\ldots,\eta_{N-1},b_0^\prime,\ldots,b_r^\prime]$. 

Let's denote $W^-\equiv[\xi_{N-1},\ldots,\xi_0]$ and $V^-\equiv[\eta_{N-1},\ldots,\eta_0]$. From the assumption of the lemma, we have
\begin{align*}
f_{\xi_0}\circ\ldots\circ f_{\xi_{N-1}}(I)\cap f_{\eta_0}\circ\ldots\circ f_{\eta_{N-1}}(I)&=\emptyset,\\
f_{\circ W^-}(I)\cap f_{\circ V^-}(I)&=\emptyset.
\end{align*}
Thus also
\begin{align*}
f_{\circ W^-}\circ f_{b_0}\circ \ldots\circ f_{b_r}(I)\cap f_{\circ V^-}\circ f_{b_0^\prime}\circ\ldots\circ f_{b_r^\prime}(I)=\emptyset.
\end{align*}
Since $f_i$ for $i\in\mathbb{N}_0$ are injective
\begin{align*}
f_{a_0}\circ\ldots\circ f_{a_l}\circ f_{\circ W^-}\circ f_{b_0}\circ \ldots\circ f_{b_r}(I)\cap f_{a_0}\circ\ldots\circ f_{a_l}\circ f_{\circ V^-}\circ f_{b_0^\prime}\circ\ldots\circ f_{b_r^\prime}(I)=\emptyset.
\end{align*}
%Thus, it does not happen that the point $x$ belongs to both sets, 
Hence the point $x$ does not belong to both sets, which is a contradiction with the fact that $C$ and $\tilde C$ are from the set $\Sigma_x^n$. 
\end{proof}
\begin{proof}[Proof of Lemma \ref{DMprop31}]
It is sufficient to combine the above statements:
\begin{align*}
&\ppp(S_{nN}^x)\stackrel{(a)}{=}\sum_{C\in \Sigma_{x}^n}\ppp(C)\stackrel{(b)}{\leq}\sum_{C\in \Sigma_{x}^n}\ppp(H_n(C))\stackrel{(c)}{=}\\
&\stackrel{(c)}{=}\ppp\left(\bigcup_{C\in\Sigma_x^n}H_n(C)\right)\stackrel{(d)}{\leq}\ppp\left( \bigcup_{C\in E^n} C\right)\stackrel{(e)}{=}
\ppp(\Sigma_n^W). 
\end{align*}
The equality $(a)$ follows from~(\ref{SnNxDecomp}), the inequality $(b)$ follows from Claim \ref{claim0}, the equality $(c)$ from injectivity, i.e. from Claim \ref{claim2}, the inequality~$(d)$ from the fact that $H_n$ maps the elements from $\Sigma_x^n$ to the set $E^n$ and the equality~$(e)$ from~(\ref{SigmanWDecomp1}).
\end{proof}

\end{document}